\documentclass[12pt,reqno]{amsart}

\usepackage{amsmath}
\usepackage{amssymb}
\usepackage{latexsym}
\usepackage{amsfonts,setspace}
\usepackage{fullpage}
\usepackage{blkarray}
\usepackage{hhline}
\usepackage{indentfirst}
\usepackage{algorithmic}
\usepackage{algorithm}
\usepackage{color}
\usepackage{changepage}
\usepackage{paralist}
\usepackage{mathtools}
\usepackage{tikz}
\usetikzlibrary{3d}
\usepackage{url}
\usepackage{subcaption}
\usepackage{bm}
\usepackage[flushleft]{threeparttable}
\usepackage{ifthen}
\usepackage{tikz-3dplot}

\usepackage{nameref,hyperref}
\usepackage[capitalize]{cleveref}

\usepackage{todonotes}

\usepackage{makecell}
\usepackage{pict2e}
\usepackage{amsthm}

\theoremstyle{definition}
\newtheorem{thm}{Theorem}
\newtheorem{rmk}[thm]{Remark}

\newtheorem{cor}[thm]{Corollary}
\newtheorem{lem}[thm]{Lemma}

\newtheorem{exm}[thm]{Example}
\newtheorem{defi}[thm]{Definition}
\newtheorem{ass}[thm]{Assumption}{\bf}{\rm}

\newcommand{\CC}{\mathbb{C}}

\newcommand{\bb}{\mathbf{b}}

\DeclareMathOperator{\im}{im}

\newcommand{\mainfilecheck}[1]{0}
\newcommand{\opA}{\mathcal A}

\newcommand{\spanz}{\mathrm{span}}
\newcommand{\rank}{\mathrm{rank}}

\normalbaselines

\newcommand{\x}{x}

\newcommand{\f}{f}
\newcommand{\g}{g}
\newcommand{\Jac}{D\f(\x)}

\newcommand{\brx}{\kappa}
\newcommand{\diag}{\mathrm{diag}}

\newcommand{\xxi}{\xi}

\newcommand{\acc}{\xi}

\newcommand{\exact}{\xi}
\newcommand{\bigO}{\mathcal{O}}

\newcommand{\opB}{\mathcal B}

\title{Two-step Newton's method for deflation-one singular zeros of analytic systems\footnote{This article is part of the volume titled 
     ``Computational Algebra and Geometry:
     A special issue in memory and honor of Agnes Szanto''.}}

\begin{document}

\begin{abstract}
We propose a two-step Newton's method for refining an approximation of a singular zero whose deflation process terminates after one step, also known as a deflation-one singularity. Given an isolated singular zero of a square analytic system, our algorithm exploits an invertible linear operator obtained by combining the Jacobian and a projection of the Hessian in the direction of the kernel of the Jacobian. We prove the quadratic convergence of the two-step Newton method when it is applied to an approximation of a deflation-one 
 singular zero. Also, the algorithm requires a smaller size of matrices than the existing methods, making it more efficient. We demonstrate examples and experiments to show the efficiency of the method.
\end{abstract}

\author{Kisun Lee}
\address{School of Mathematical and Statistical Science, Clemson University, 220 Parkway Drive, Clemson, SC 29634, USA}
\email{kisunl@clemson.edu}

\author{Nan Li}
\address{School of Mathematical Sciences, Shenzhen University, Shenzhen 518060, Guangdong, China \& Guangdong Key Laboratory of Intelligent Information Processing, Shenzhen 518060, Guangdong, China}
\email{nan.li@szu.edu.cn}

\author{Lihong Zhi}
\address{Key Laboratory of Mathematics Mechanization, Academy of Mathematics and Systems Science, Chinese Academy of Sciences, Beijing 100190, China \& University of Chinese Academy of Sciences, Beijing 100049, China}
\email{lzhi@mmrc.iss.ac.cn}
\thanks{This research is supported by the National Key Research
 Project of China (2018YFA0306702), the National Natural Science
 Foundation of China (12071467, 12171324), the Basic and Applied Basic Research
Foundation of Guangdong Province (2022A1515010811), the Guangdong Provincial Pearl River Talents Program (2021QN02X310).}

\keywords{
    deflation-one singularity, local dual space, Newton's method, quadratic convergence}

\maketitle
\makeatletter
\newcommand{\keywords}[1]{%
\let\@@oldtitle\@title%
\gdef\@title{\@@oldtitle\footnotetext{\emph{Key words and phrases.} #1.}}%
}
\makeatother

\section{Introduction}\label{intr}


Consider a square analytic system $f=[f_1,\ldots,f_n]^\top$ and a zero $\xi\in\mathbb{C}^n$ of $f=0$.  Let $Df(\xxi)$ denote the Jacobian matrix of $f$  evaluated at $\xi$. When $Df(\xi)$ is invertible, the \textit{Newton iteration} $N(f, x) = x - Df(x)^{-1}f(x)$, applied to an approximation $x$ that is close enough to $\xi$, \textit{converges quadratically} to $\xi$. Namely,
$\|N(f,x)-\acc  \| =\bigO( \| x -\acc   \|^2)$ where $\bigO(g)$ indicates that the value is bounded above by $g$ up to a positive constant.

   When $Df(\xxi)$ is not invertible, i.e., $\dim\ker Df(\xxi) \geq 1$, Newton's method fails to maintain local quadratic convergence because there may exist a positive-dimensional manifold near $\xi$ that satisfies $\det Df(\xi)=0$  \cite{Reddien:1978}.
   Various modifications have been proposed to restore the quadratic convergence of Newton's method for isolated singular zeros from different perspectives. 

In cases where $\dim\ker Df(\xi) = 1$, a two-step Newton's method was proposed in \cite{LZ:2011} for refining an approximate zero $x$ such that it converges quadratically to $\xi$. The first step of the iteration involves projecting $x$ to $x'$ such that $x'-\xi$ approximately belongs to the one-dimensional linear space $\mathrm{span}_{\mathbb{C}}\{v'\}$, which approximately coincides with $\ker Df(\xi)$.
The second step of the iteration estimates a step length $\delta$ by solving a sequence of least squares problems and a linear system such that the resulting approximate solution  $x''=x'+\delta\cdot v'$ satisfies $  \|x''-\xi\|=\bigO(\| x -\xi\|^2)$.
In \cite{li2022improved}, an improved two-step Newton's method was proposed, which retains quadratic convergence without solving any least-squares problems or linear systems. 

This paper extends the algorithm in~\cite{li2022improved} to the case of  \textit{deflation-one singular zeros}, i.e.,  the deflation algorithm proposed by   Leykin, Verschelde, and Zhao ~\cite{leykin2006newton}  terminates in one iteration. 
We present a two-step Newton's method for refining the deflation-one singular zeros and prove its quadratic convergence.

One possible application of Newton's method for singular zeros is the \textit{zero cluster isolation problem}, which considers constructing a region that separates a cluster of points approximating a singular zero of the system from other zeros. This problem has been studied for various forms of singular zeros in works such as \cite{Dedieu2001on, HJLZ2020,burr2021inflation,burr2023isolating}. As these methods often require a well-approximated point as input, our two-step Newton's method can provide such a point.

Note that our computations are limited to a small enough neighborhood of one point and can allow small perturbations of functions within that neighborhood. Hence, it is possible to substitute analytic functions with polynomials provided there is a reliable method for estimating the difference between the two (see Example 9 (\ref{running_example_2})). For this reason, we will focus on the polynomial case for the rest of the paper.

\subsection*{\bf{Outline of results}}\,
Let $f=[f_1,\ldots,f_n]^\top\in \mathbb{C}[X_1, \ldots, X_n]^n$ be a square polynomial system with an isolated singular zero $\xi$. In Theorem~\ref{thm:multiplerootCharacterization}, we show that $\xi$ is a deflation-one singular solution if and only if
the linear operator 
$$\mathcal{A}(\xi)=Df(\xi)+D^2f(\xi)(v,\Pi_{\ker Df(\xi)}\cdot)$$
is invertible for almost all choices of $v\in \ker Df(\xi)$, where $\Pi_{\ker Df(\xi)}$ is the Hermitian projection to $\ker Df(\xi)$.
We can apply the singular value decomposition to $Df(\xi)$, which is denoted by $U\cdot\Sigma\cdot V^*$, and obtain $U=[U_1,U_2]$, $V=[V_1,V_2]$, where $U_1,V_1\in \mathbb{C}^{n\times (n-\kappa)}$ and $U_2,V_2\in\mathbb{C}^{n\times \brx}$, with $\kappa$ being the corank of $Df(\xi)$.
In Theorem \ref{opB}, we show that the deflation-one singularity can be further characterized by the invertibility of a smaller matrix of size $\brx\times \brx$ $$\opB(\xi)=U_2^*\cdot D^2\f(\xi)\cdot v\cdot V_2.$$

We propose a two-step Newton's method for refining a given approximation $x$ of a deflation-one singular zero $\xi$. 
Note that when $x$ is sufficiently close to $\xi$, the singular value decomposition of $Df(\xi)$ can be approximated as $Df(x)=U\cdot \Sigma \cdot V^*$. Thus, we can compute the breadth $\brx$ by using $Df(x)$ and then verify if $\xi$ is a deflation-one singular zero by checking the invertibility of $\opA(x)$.
The method involves computing the singular value decomposition $Df(x)=U\cdot \Sigma\cdot V^*$ and applying the following steps:
\begin{enumerate}
\item Refine $\x$ to 
\[
  \x'=\x-V_1\cdot\Sigma_1^{-1}\cdot U_1^*\cdot\f(\x),
\]
where $\Sigma_1=\diag(\sigma_1,\ldots,\sigma_{n-\kappa})$ is a diagonal matrix consisting of the first $n-\kappa$ singular values of $Df(x)$.
\item Refine $\x'$ to
\[\x''=\x' + V_2 \cdot \boldsymbol{\delta},\]
where $\boldsymbol{\delta}$ is obtained by solving the linear equation $\opB'\cdot\boldsymbol{\delta}=-U_2^*\cdot D\f(\x')\cdot v$ with $\opB'=U_2^*\cdot D^2\f(\x')\cdot v\cdot V_2$. Here, note that $U_2^*,v$ and $V_2$ in the definition of $\opB'$ are taken from $Df(x)$ not $Df(x')$.
\end{enumerate}
We prove the quadratic convergence of the iterations when the approximation $x$ is sufficiently close to the exact deflation-one singular solution $\xi$.

\subsection*{\bf Related works}\, There are many different numeric and symbolic approaches to compute singular zeros of polynomial systems.   In~\cite{Rall66}, Rall studied the convergence
property of Newton's method for singular solutions, and many
modifications of Newton's method to restore the quadratic
convergence for singular solutions have been proposed
in ~\cite{Rall66,Reddien:1978,Reddien:1980,DeckerKelley:1980I,DeckerKelley:1980II,GriewankOsborne:1981,DeckerKelley:1982}.

In~\cite{Griewank85},  Griewank constructed a bordered system  from the initial system $f$ and the singular value decomposition of the Jacobian matrix $Df(x)$ to restore
the quadratic convergence of Newton's method when the Jacobian at the exact zero has corank
one. The method was  extended by Shen and Ypma~\cite{ShenYpma05,ShenYpma2007} to the case where the Jacobian has an arbitrary high-rank deficiency.

In~\cite{OWM:1983,YAMAMOTONORIO:1984,Ojika:1987}, Ojika  et al. proposed a deflation method to construct a regular system to refine an approximate isolated singular solution to high
accuracy. 
In~\cite{Lecerf:2002}, Lecerf
  gave a deflation algorithm that outputs a regular triangular system at the singular solution. In \cite{giusti2005location,giusti2007location}, Giusti, Lecerf, Salvy and Yakoubsohn established the criterion for detecting clusters and convergence analysis for Newton's method for one variable and singularity of embedding dimension one cases.
The deflation method has been further developed and generalized by Leykin, Verschelde, and Zhao~\cite{leykin2006newton, LVZ:2008} for singular solutions whose Jacobian matrix has arbitrary high-rank deficiency and for overdetermined polynomial systems.
Furthermore, they proved that the number of deflations needed to
derive a regular solution of an augmented system is strictly less
than the multiplicity. Dayton and Zeng \cite{DZ:2005,DLZ:2009} proved that
the depth of the local dual space is a tighter bound for the
number of deflations.  
In \cite{MM:2011}, Mantzaflaris and Mourrain proposed a one-step
deflation method and verified a multiple zero of a nearby system with
a given multiplicity structure
which depends on the accuracy of the
given an approximate multiple zero. 
Hauenstein, Mourrain and Szanto 
proposed a novel deflation method that extends their early works \cite{MM:2011,AHZ2018} to verify the existence of an isolated singular zero with a given multiplicity structure up to a given order ~\cite{HMS2015,HMS2017}. In~\cite{giusti2020approximation}, Giusti and Yakoubsohn proposed a new deflation sequence using the kernel operator defined by the Schur complement of the Jacobian matrix and proved a new $\gamma$-theorem for analytic regular systems. 
In \cite{MMS2020,mantzaflaris2023certified},  Mantzaflaris,  Mourrain, and Szanto proposed a certified iteration method for computing isolated singular roots.
From the perspective of  polynomial homotopy continuation, Verschelde and Viswanathan \cite{verschelde2022locating} considered a way to deal with singularity while tracking a homotopy path.
More recently, the technique of inflation has also been introduced for separating singular zeros in \cite{burr2021inflation,burr2023isolating}.

\paragraph*{\bf Structure of the paper}\quad
Section \ref{sec2} recalls some definitions and introduces the deflation method given in \cite{leykin2006newton} for refining singular solutions.
 In Section \ref{sec3}, we characterize more properties of deflation-one singularity. Section \ref{sect4} proposes the two-step Newton's method for such singular zeros.
In Section \ref{sect5},
we compare the performance of our algorithm with the algorithm in \cite{leykin2006newton} for a list of benchmark examples.

\section{Preliminaries}\label{sec2}

In this section, we introduce some symbolic tools and numerical algorithms for analyzing isolated singular zeros of polynomial systems. 
We begin with the definition of an isolated singular zero.

\begin{defi}\label{isolatedsingularzero}
We say that $\xi\in\mathbb{C}^n$ is an \textit{isolated singular zero} of a polynomial system $f=[f_1,\ldots,f_n]^\top\in \mathbb{C}[X_1, \ldots, X_n]^n$, if
\begin{enumerate}
  \item $\xi$ is an isolated zero of $f$, i.e., there is $r>0$ such that $\xi$ is the only zero of $f$ in $\mathrm{Ball}(\xi,r)$,
  \item $Df(\xi)$ is singular, i.e., $\dim\ker Df(\xi)>0$.
\end{enumerate}
\end{defi}
We introduce two key tools for analyzing isolated singular zeros of polynomial systems: the \textit{local dual space} and the \textit{deflation} method. The local dual space is a hybrid (symbolic-numeric) tool used for computing the multiplicity structure of an isolated singular zero, while the deflation method is a numerical algorithm for reducing the multiplicity of an isolated singular zero. For the deflation method, we particularly focus on an algorithm established in \cite{leykin2006newton}.

\subsection{Local dual space}

The local dual space is a powerful tool to characterize the multiplicity structure of an isolated singular zero for a polynomial system.
Given a sequence $\boldsymbol{\alpha}=(\alpha_1,\dots,\alpha_n)\in \mathbb{N}^n$ and a point $\xi\in\mathbb{C}^n$, we define a differential functional $\mathbf{d}^{\boldsymbol{\alpha}}_{\xi}: \CC[X_1,\dots, X_n] \rightarrow \mathbb{C}$ by
\begin{equation*}
	\mathbf{d}^{\boldsymbol{\alpha}}_{\xi}(g)=\frac{1} {\alpha_1!\cdots
		\alpha_n!}\cdot\frac{\partial^{|\boldsymbol{\alpha}|} g}{\partial
		X_1^{\alpha_1}\cdots \partial X_n^{\alpha_n}}(\xi),\quad\forall g\in \CC[X_1,\dots, X_n]
\end{equation*}
where $|\boldsymbol{\alpha}|=\alpha_1+\cdots+\alpha_n$. To simplify notation, we will often omit $\xi$ when it is clear from context, and use $d_i$ instead of $\frac{\partial}{\partial X_i}$. 
The local dual space of a polynomial system $\f$ at an isolated zero $\xi$ is a subspace of
$\mathfrak{D}_{\xi}=\spanz_\mathbb{C}\{\mathbf{d}^{\alpha}_{\xi}\}$ which is defined by
$\mathcal{D}_{f,\xi}=\{\Lambda\in
\mathfrak{D}_{\xi}\ |\  \Lambda(g)=0, ~\forall g\in I_f\}$
where $I_f$ is the ideal generated by polynomials in $f$. 
In other word, $\Lambda\in\mathcal{D}_{f,\xi}$ if and only if $\Lambda(h^\top\cdot f)=0$ for all $h\in \CC[X_1,\dots, X_n]^n$.
This condition is called the \textit{closedness condition} or the \textit{stability condition} (see \cite[Chapter 8]{stetter2004numerical}).

Let $\mathcal{D}_{f,\xi}^{(k)}$ be the subspace of $\mathcal{D}_{f,\xi}$ containing differential functionals of order at most $k$. We define three constants:
\begin{enumerate}
  \item $\brx=\dim\left(\mathcal{D}_{f,\xi}^{(1)}\right)-\dim\left(\mathcal{D}_{f,\xi}^{(0)}\right)$ (\textit{breadth}),
  \item $\rho=\min\left\{k\mid \dim\left(\mathcal{D}_{f,\xi}^{(k+1)}
      \right)=\dim\left(\mathcal{D}_{f,\xi}^{(k)}
      \right)\right\}$ (\textit{depth}),
  \item $\mu=\dim\left(\mathcal{D}_{f,\xi}^{(\rho)}\right)$ (\textit{multiplicity}).
\end{enumerate}
From the definition, it is clear that $\xi$ is an isolated singular zero of $f$ if and only if $1\leq\brx=\dim\ker Df(\xi)\leq n$ and $0<\rho<\mu<\infty$.

Let $\mathcal{C}_{f,\xi}^{(k)}$ be the subspace of $\mathfrak{D}_{\xi}$ consisting of functionals  $\Lambda$ such that $\Phi_{i}(\Lambda)\in\mathcal{D}_{f,\xi}^{(k-1)}$ for $i=1,\ldots,n$, where $\Phi_i$ is a linear operator defined by $\Phi_i(\mathbf{d}^{\boldsymbol{\alpha}}_{\xi})=\mathbf{d}^{\boldsymbol{\alpha}-\mathbf{e}_i}_{\xi}$ if $\alpha_i>0$, and zero otherwise. Note that $\mathcal{C}_{f,\xi}^{(k)}$ is a closed subspace of $\mathfrak{D}_{\xi}$ consisting of functionals with differentials of order at most $k$. By the closedness condition, we have $\Lambda\in\mathcal{D}_{f,\xi}$ if and only if
$\Phi_{i}(\Lambda)\in\mathcal{D}_{f,\xi}$ for all $i=1,\ldots,n$ and $\Lambda(f)=0$.
Therefore, it is straightforward to get $\mathcal{D}_{f,\xi}^{(k-1)}\subset\mathcal{D}_{f,\xi}^{(k)}\subset\mathcal{C}_{f,\xi}^{(k)}$.

By using $\mathcal{C}_{f,\xi}^{(k)}$ as an intermediate step, it is possible to compute a basis for $\mathcal{D}_{f,\xi}$ in a recursive manner
\[\mathcal{D}_{f,\xi}^{(0)}\nearrow\mathcal{C}_{f,\xi}^{(1)}\searrow\mathcal{D}_{f,\xi}^{(1)}\nearrow\cdots\searrow\mathcal{D}_{f,\xi}^{(\rho)}\nearrow\mathcal{C}_{f,\xi}^{(\rho+1)}\searrow\mathcal{D}_{f,\xi}^{(\rho+1)}\]
where $\nearrow$ and $\searrow$ indicate the process of solving certain linear systems. 
To clarify the meaning of $\searrow$,
it corresponds to solving a linear system to find a basis of $\mathcal{D}_{f,\xi}^{(k)}$ given a basis of $\mathcal{C}_{f,\xi}^{(k)}$.
For example, if we have $\mathcal{C}_{f,\xi}^{(k)}=\spanz_\mathbb{C}\{\Lambda_1,\ldots,\Lambda_{m_k}\}$, then $\mathcal{D}_{f,\xi}^{(k)}$ is isomorphic to the kernel of the matrix $C=[\Lambda_1(f)~\cdots~\Lambda_{m_k}(f)]$.
Trivially, $\rank \, C = \dim\,\mathcal{C}^{(k)}_{\f,\xi}-\dim\,\mathcal{D}^{(k)}_{\f,\xi}\leq n$.

\subsection{Newton's method with deflation}

Deflation is a method to desingularize isolated singular zeros by introducing additional equations and variables to a given system, generating an augmented system with a ``deflated'' multiplicity structure at the isolated singular zero (see \cite{leykin2006newton} for instance). Suppose we are given an isolated singular zero $\xi\in\mathbb{C}^n$ of a polynomial system $\f$ with breadth $\brx$. To construct an augmented system with a new isolated zero, we choose a random matrix $B\in\mathbb{C}^{n\times (n-\brx+1)}$ and a random vector $\bb\in\mathbb{C}^{n-\brx+1}$ such that there exists a unique vector $\boldsymbol{\lambda}=[{\lambda}_1,{\lambda}_2,\dots,{\lambda}_{n-\brx+1}]^\top$ satisfying $(\xi,\boldsymbol{\lambda})$ as an isolated zero of the new system 
\begin{equation}\label{LVZ_deflation}
\begin{bmatrix}
             \f \\
             D\f\cdot B\cdot\boldsymbol{\lambda}\\
             \bb^\top\cdot\boldsymbol{\lambda}-1
           \end{bmatrix}
.
\end{equation}
If $(\xi,\boldsymbol{\lambda})$ remains singular, 
the process is iterated for the augmented system (\ref{LVZ_deflation}) and the augmented zero $(\xi,\boldsymbol{\lambda})$.

When an approximation $\x$ of the singular zero $\xi$ is given, the breadth $\brx$ can be determined by solving SVD with truncating singular values within a tolerance $\tau$, and the vector $\boldsymbol{\lambda}$ can be determined by solving the least square problem. The resulting zero $(x,\boldsymbol{\lambda})$ approximates the exact zero. We refer Algorithm 2 in \cite{leykin2006newton} as LVZ algorithm.

It is known that LVZ algorithm terminates after finitely many steps if $\xi$ is isolated and $\x$ is sufficiently close to $\xi$. In \cite[Theorem 3.1]{leykin2006newton}, it is proved that the number is less than the multiplicity $\mu$. In \cite[Theorem 3]{DZ:2005}, Dayton and Zeng improved the bound to at most $\rho$. Li and Zhi in \cite[Theorem 3.8]{LZ:2011} studied the case of $\kappa=1$ and proved that the number of iterations is always equal to  $\mu-1$. Based on the benchmark examples in \cite{DZ:2005}, we have observed that for many types of polynomial systems with an isolated singular zero, the LVZ algorithm terminates after only one deflation step. Once the deflation process terminates, then Newton's method can be applied to the system (\ref{LVZ_deflation}) to refine the approximation $(x,\boldsymbol{\lambda})$. 

\begin{rmk} \label{rmk:deflationProperties2} We note that when applying Newton's method to the overdetermined system obtained by the deflation algorithm, convergence to a stationary point rather than a global minimum is possible. For example, for an equation $f=x^2$, the augmented system $[x^2,2x]^\top$ obtained by $f$ and $Df$ has the origin as the only regular zero. On the other hand, since $\|[x^2,2x]^\top\|^2=x^4+4x^2$, a stationary point $x=2i$ exists that is not a zero of the augmented system. This type of convergence to a stationary point cannot be avoided, regardless of the distance between the approximation and the isolated zero.
\end{rmk}

\section{Characterizing deflation-one singularity}\label{sec3}

We are interested in isolated singular zeros such that LVZ algorithm terminates by one iteration, which we call deflation-one singularities. Our goal is to characterize such singularities to establish Newton's method to refine them.


We propose an augmented system equivalent to the system (\ref{LVZ_deflation}).
Let $V=[V_1,V_2]\in\mathbb{C}^{n\times n}$ be an invertible matrix satisfying
\begin{enumerate}
  \item $V_1\in\mathbb{C}^{n\times (n-\brx)}$, $\im V_1=\{\ker Df(\xi)\}^{\bot}$,
  \item $V_2\in\mathbb{C}^{n\times \brx}$, $\im V_2=\ker Df(\xi)$.
\end{enumerate} Let $\boldsymbol{\lambda}_2=[\lambda_1,\ldots,\lambda_{\brx}]^\top\in\mathbb{C}^{\brx}$ be a random vector and $\boldsymbol{\lambda}_1=[\lambda_{\kappa+1},\dots, \lambda_n]^\top$ be a vector consisting of $n-\kappa$ variables. Then, there is an augmented system with $2n-\kappa$ variables (that is, $X_1,\dots, X_n,\boldsymbol{\lambda}_1$) and $2n$ equations
\begin{equation*}
\g=\begin{bmatrix}
             \f \\
             D\f\cdot V\cdot\boldsymbol{\lambda}
           \end{bmatrix}
\end{equation*}
where $\boldsymbol{\lambda}=[\boldsymbol{\lambda}_1;\boldsymbol{\lambda}_2]$. Furthermore, $(\xi,0)$ is an isolated zero of $g$.
The Jacobian matrix of $\g$ evaluated at $(\xi,0)$ is 
\begin{equation}\label{deflationM}
Dg(\xi,0)=\begin{bmatrix}
                        Df(\xi) & ~~\mathbf{0} \\
                        D^2f(\xi)\cdot V_2\cdot\boldsymbol{\lambda}_2 & ~~Df(\xi)\cdot V_1
                      \end{bmatrix}\in\mathbb{C}^{2n\times (2n-\brx)},
\end{equation}
where $D^2f(\xi)$ is the $n\times n\times n$ tensor consisting of all second order derivatives of $\f$ at $\xi$.
Therefore, $\xi$ is of deflation-one singularity for $\f$ if and only if $Dg(\xi,0)$ is of full column rank for almost all choices of $\boldsymbol{\lambda}_2\in\mathbb{C}^{\brx}$. We will explain the words  ``almost all''   later in this section. The following example gives an intuitive description of the deflation-one criterion.

\begin{exm}\cite{kobayashi1998numerical}\label{running_example_1}
  The system
  \begin{equation*}
  f(x,y,z) = \begin{bmatrix}
  x^2-x+y+z-2\\
  y^2+x-y+z-2\\
  z^2+x+y-z-2
  \end{bmatrix}
  \end{equation*}
  has an isolated singular zero at $\xi=[1,1,1]^\top$ with $\kappa =2,\rho=2$ and $\mu =4$.
  Then, we have $$Df(\xi)=\begin{bmatrix}
                        1 & 1 & 1 \\
                        1 & 1 & 1 \\
                        1 & 1 & 1
                      \end{bmatrix}.$$ 
                   We choose $V=[V_1,V_2]$, where  $V_1=\begin{bmatrix}
                        1 \\
                        1 \\
                        1
                      \end{bmatrix}$ and  $V_2=\begin{bmatrix}
                        1  & 1\\
                        -1 & 0\\
                        0  & -1
                      \end{bmatrix}$  for the deflation.
 If we set $\boldsymbol{\lambda}_2=\begin{bmatrix}
                        1  \\
                        1
                      \end{bmatrix}$, then $(\xi,0)$ is an isolated regular zero of the augmented system
  \[\g=\begin{bmatrix}
             \f \\
             D\f\cdot V\cdot \begin{bmatrix}
                 \lambda\\
                 \boldsymbol{\lambda}_2
             \end{bmatrix}
           \end{bmatrix}=\begin{bmatrix}
             \f \\
             2x\lambda+4x+\lambda-4 \\
             2y\lambda-2y+\lambda+2 \\
             2z\lambda-2z+\lambda+2
           \end{bmatrix},\]
and its Jacobian matrix
  \[Dg(\xi,0)=\begin{bmatrix}
                        Df(\xi) & ~~\mathbf{0} \\
                        D^2f(\xi)\cdot V_2\cdot\boldsymbol{\lambda}_2 & ~~Df(\xi)\cdot V_1
                      \end{bmatrix}=\begin{bmatrix}
                        1 & 1 & 1 & 0\\
                        1 & 1 & 1 & 0\\
                        1 & 1 & 1 & 0\\
                        4 & 0 & 0 & 3\\
                        0 & -2 & 0 & 3\\
                        0 & 0 & -2 & 3
                      \end{bmatrix}\]
    is of full column rank.  Therefore, $\xi$ is a deflation-one singular zero of $f$.
\end{exm}

For almost all choices of $v\in \ker Df(\xi)$, the following theorem shows that the existence of an invertible linear operator is equivalent to a deflation-one singularity.

\begin{thm}\label{thm:multiplerootCharacterization}
Let $\xi\in\mathbb{C}^n$ be a deflation-one singular zero of a polynomial system $f\in\mathbb{C}[X_1,\dots, X_n]^n$, 
then $Dg(\xi,0)$ is of full column rank for almost all choices of $\boldsymbol{\lambda}_2\in\mathbb{C}^{\brx}$ if and only if the linear operator
\begin{equation}\label{OperatorA}
\mathcal{A}(\xi)=Df(\xi)+D^2f(\xi)(v,\Pi_{\ker Df(\xi)}\cdot)
\end{equation}
is invertible for almost all choices of $v\in \ker Df(\xi)$, where $\Pi_{\ker Df(\xi)}$ is the Hermitian projection to $\ker Df(\xi)$.
\end{thm}
\begin{proof}
Consider the singular value decomposition 
$Df(\xi)=U\cdot\Sigma\cdot V^*$. Assume $U=[U_1,U_2]$, $V=[V_1,V_2]$, where  $U_2,V_2\in\mathbb{C}^{n\times \brx}$,   $\im V_1=\{\ker Df(\xi)\}^{\bot}$ and   $\im V_2=\ker Df(\xi)$. Furthermore, we let $v=V_2\cdot\boldsymbol{\lambda}_2\in\ker Df(\xi)$.  Then, we have
\begin{equation*}
	\mathcal{A}(\xi)=Df(\xi)+D^2f(\xi)\cdot v \cdot V_2 \cdot V_2^*.
\end{equation*}
  Let $w=[
  w_1,
  w_2
 ]^\top\in\mathbb{C}^{n}$ where $w_1\in\mathbb{C}^{\brx}$, $w_2\in\mathbb{C}^{n-\brx}$. Then, we have
  \begin{align*}
  &Dg(\xi,0)\text{ is of full column rank},\\
  \Leftrightarrow~& Dg(\xi,0)\begin{bmatrix}
  V_2\cdot w_1 \\
  w_2
  \end{bmatrix}
  =0\mbox{ implies }w=0, \\
  \Leftrightarrow~& (D^2f(\xi)\cdot V_2\cdot\boldsymbol {\lambda}_2)\cdot(V_2\cdot w_1)+Df(\xi)\cdot V_1 \cdot w_2=0\mbox{ implies }w=0, \\
  \Leftrightarrow~&  \mathcal{A}(\xi)\cdot V \cdot w=0\mbox{ implies }w=0,\\
  \Leftrightarrow~& \mathcal{A}(\xi)\mbox{ is invertible}.
  \end{align*}  It concludes the desired result. 
\end{proof}

The above theorem confirms that the deflation-one singularity can be characterized by the invertibility of an $n \times n$ matrix  $\opA(\xi)$ instead of the full-column rank condition of the $2n\times (2n-\brx)$ Jacobian matrix  $Dg(\xi,0)$. 
The following theorem reveals that the invertibility of a smaller matrix can be used to characterize the deflation-one singularity.

\begin{thm}\label{opB}
Consider the singular value decomposition $Df(\xi)=U\cdot \Sigma \cdot V^*$ with $U=[U_1,U_2]$, $V=[V_1,V_2]$, where  $U_2,V_2\in\mathbb{C}^{n\times \brx}$,   $\im V_1=\{\ker Df(\xi)\}^{\bot}$ and   $\im V_2=\ker Df(\xi)$.
Then for $v\in \ker Df(\xi)$, a $\kappa\times \kappa$ matrix
\begin{equation}\label{OperatorB}
	\opB(\xi)=U_2^*\cdot D^2\f(\xi)\cdot v\cdot V_2.
\end{equation}
is invertible if and only if $\opA(\xi)$ is invertible.
\end{thm}
\begin{proof}
It is a straightforward result since
$$U^{\ast}\cdot \mathcal{A}(\xi)\cdot V=\left[
                        \begin{array}{c}
                          U_1^{\ast} \\
                          U_2^{\ast}
                        \end{array}
                      \right]\cdot\left(Df(\xi)+D^2f(\xi)\cdot v \cdot V_2\cdot V_2^*\right)\cdot[V_1,V_2]=\left[
                        \begin{array}{cc}
                          \Sigma_1~~ & U_1^*\cdot D^2\f(\xi)\cdot v\cdot V_2 \\
                          \boldsymbol{0}~~ & \opB(\xi) \\
                        \end{array}
                      \right],
$$
where $\Sigma_1=\diag(\sigma_1,\ldots,\sigma_{n-\brx})$ consists of the first $n-\kappa$ singular values of $Df(\xi)$.
\end{proof}

\begin{exm}
[Example \ref{running_example_1} continued] Consider the singular value decomposition
\renewcommand\arraystretch{1.5}
\[Df(\xi)=U\cdot\Sigma\cdot V^*=\begin{bmatrix}
                        -\frac{\sqrt{3}}{3} & \frac{\sqrt{2}}{2} & \frac{\sqrt{6}}{6} \\
                        -\frac{\sqrt{3}}{3} & 0 & -\frac{\sqrt{6}}{3} \\
                        -\frac{\sqrt{3}}{3} & -\frac{\sqrt{2}}{2} & \frac{\sqrt{6}}{6}
                      \end{bmatrix} \cdot\begin{bmatrix}
                        3& 0 & 0 \\
                        0&0&0\\
                        0&0&0
                        \end{bmatrix}\cdot
                       \begin{bmatrix}
                        -\frac{\sqrt{3}}{3} & -\frac{\sqrt{3}}{3} & -\frac{\sqrt{3}}{3} \\
                        \frac{\sqrt{2}}{2} & 0 & -\frac{\sqrt{2}}{2} \\
                        -\frac{\sqrt{6}}{6} & \frac{\sqrt{6}}{3} & -\frac{\sqrt{6}}{6}
                      \end{bmatrix}\]
                  of $Df(\xi)$.
    Choosing $v=[2,-1,-1]^\top$, we have
     \renewcommand\arraystretch{1.5} 
    \[\mathcal{A}(\xi)=\begin{bmatrix}
                        \frac{11}{3} & -\frac{1}{3} & -\frac{1}{3} \\ 
                        \frac{5}{3} & -\frac{1}{3} & \frac{5}{3} \\
                        \frac{5}{3} & \frac{5}{3} & -\frac{1}{3}
                      \end{bmatrix},~\mathcal{B}(\xi)=\begin{bmatrix}
                        1 & -\sqrt{3} \\
                        \sqrt{3} & 1
                      \end{bmatrix}. \]
  The matrices  $\mathcal{A}(\xi), \mathcal{B}(\xi)$   are both invertible. Hence, $\xi$ is a deflation-one singular zero of $f$.
\end{exm}
The sizes of the matrices $Dg(\xi,0)$, $\mathcal{A}(\xi)$, and $\mathcal{B}(\xi)$ decrease from $6\times 4$, $3\times 3$ to $2\times 2$, respectively. Therefore, using $\mathcal{B}(\xi)$ would result in a more efficient algorithm for refining a singularity.

\begin{rmk}\label{tolerance_qualitative}
Since the exact singular solution $\xi$ is unknown,  the two-step Newton's method starts with an approximation $x$. Assuming $x$ is sufficiently close to $\xi$, the singular value decomposition of $Df(x)$ is used to estimate the breadth $\brx$ and approximate the kernel of $Df(\xi)$, and  to construct the matrix $\opB(x)$. Note that $\opB(x)$ remains 
invertible as long as $x-\xi$ is sufficiently small and $\brx$ is accurately estimated.
\end{rmk}

The invertibility of $\opB(\xi)$ also clarifies the meaning of ``for almost all choices of $v\in \ker Df(\xi)$'' and ``for almost all choices of $\boldsymbol{\lambda}_2\in\mathbb{C}^{\brx}$'' in Theorem \ref{thm:multiplerootCharacterization}.
Suppose $v=V_2\cdot\boldsymbol{\lambda}_2\in\ker Df(\xi)$, then we can view $h(\boldsymbol{\lambda}_2)=\det \opB(\xi)$ as a homogeneous polynomial in the variables $\lambda_1,\ldots,\lambda_{\brx}$. It is evident that $\opB(\xi)$ is invertible if and only if $\boldsymbol{\lambda}_2\in\mathbb{C}^{\brx}\setminus h^{-1}(0)$. Hence, if $h$ is not identically zero, then $h^{-1}(0)$ is an algebraic variety of dimension $\brx-1$, implying that $\opB(\xi)$ is invertible for almost all choices of $v\in \ker Df(\xi)$. Thus, if $h$ is not identically zero, $\opA(\xi)$ is invertible for almost all choices of $v\in \ker Df(\xi)$, such that $Dg(\xi,0)$ is of full column rank for almost all choices of $\boldsymbol{\lambda}_2\in\mathbb{C}^{\brx}$.
Based on the above analysis, one necessary condition of deflation-one singularity can be proposed using the local dual space $\mathcal{D}_{\f,\xi}$ and its closed subspace $\mathcal{C}_{\f,\xi}$.

\begin{cor}\label{neccesary}
If $\mathcal{B}(\xi)$ is invertible for almost all choices of $v\in \ker Df(\xi)$, then $\dim\,\mathcal{C}^{(2)}_{\f,\xi}-\dim\,\mathcal{D}^{(2)}_{\f,\xi}=n$.
\end{cor}
\begin{proof}
 Since $\dim\,\mathcal{C}^{(2)}_{\f,\xi}-\dim\,\mathcal{D}^{(2)}_{\f,\xi}\leq n$, we only need to prove that $\det \opB(\xi) \equiv0$ if $\dim\,\mathcal{C}^{(2)}_{\f,\xi}-\dim\,\mathcal{D}^{(2)}_{\f,\xi}<n$.
Since $\mathcal{C}_{f,\xi}^{(1)}=\spanz_\mathbb{C}\{1,d_{1},\ldots,d_n\}$, then $\mathcal{D}_{f,\xi}^{(1)}$ is isomorphic to the kernel of
$C_{1}=[0~Df(\xi)]$
which is the image of $\begin{bmatrix}1 & 0 \\ 0 & V_2\end{bmatrix}$. As $\brx>0$, letting $V_2=[v_1,\ldots,v_{\brx}]$,
we have that $\mathcal{D}^{(2)}_{\f,\xi}$ is isomorphic to the kernel of
\[C_{2}=\bigg[0~Df(\xi)~D^2\f(\xi)\cdot v_1 \cdot v_1~\cdots~D^2\f(\xi)\cdot v_1 \cdot v_{\brx}~\cdots~D^2\f(\xi)\cdot v_{\brx} \cdot v_{\brx}\bigg]\]
satisfying that $\rank\,C_{2} = \dim\,\mathcal{C}^{(2)}_{\f,\xi}-\dim\,\mathcal{D}^{(2)}_{\f,\xi}$.

On the other hand, $\opB(\xi)$ can be rewritten in the form of
\[\opB(\xi)=U_2^* \cdot C_2\cdot \left[
                        \begin{array}{l}
                          ~~0_{(n+1)\times\brx} \\
                          \big[\boldsymbol{\lambda}_2\big]_{\brx(\brx-1)/2\times\brx}
                        \end{array}
                      \right],\]
where $\big[\boldsymbol{\lambda}_2\big]$ is some matrix of size $\frac{\brx(\brx-1)}{2}\times\brx$ consisting of  $\lambda_1,\ldots,\lambda_{\brx}$ and $0$ such that the equality holds.

Finally, if $\dim\,\mathcal{C}^{(2)}_{\f,\xi}-\dim\,\mathcal{D}^{(2)}_{\f,\xi}<n$, then $\dim\,\ker C_2^*>0$. Moreover, since we have $\ker C_2^*\subset \ker Df(\xi)^*=\im U_2$, it follows that $\dim\,\ker (C_2^*\cdot U_2)>0$. Therefore, $\det \opB(\xi) \equiv0$.
\end{proof}

We conclude this section by presenting several examples that illustrate important properties of the deflation-one singularity.

\begin{exm}
\begin{enumerate}
	\item \label{exm:2}
The deflation terminates in one iteration doesn't mean that the zero is isolated.
	The system
	\begin{equation*}
		f(x,y) = \begin{bmatrix}
			x^2\\
			xy
		\end{bmatrix}
	\end{equation*}
	has a non-isolated singular zero $\xi$ at the origin.
Then, we have that
	$$\mathcal{C}_{f,\xi}^{(2)}=\spanz_\mathbb{C}\left\{1,d_{1},d_2,d_1^2,d_1d_2,d_2^2\right\},$$
	and $\mathcal{D}_{f,\xi}^{(2)}$ is isomorphic to the kernel of
	$C_{2}=[0_{2\times1}~0_{2\times2}~e_1~e_2~0_{2\times1}]$.
	Therefore,
	\[\opB(\xi)=C_2\cdot \left[
	\begin{array}{cc}
		0_{3\times1} & 0_{3\times1} \\
		\lambda_1 & 0 \\
		\lambda_2 & \lambda_1 \\
		0 & \lambda_2
	\end{array}
	\right]=\left[
	\begin{array}{cc}
		\lambda_1 & 0 \\
		\lambda_2 & \lambda_1
	\end{array}
	\right],\]
	so that $\det \opB(\xi) = \lambda_1^2 \not\equiv0$. Thus, the zero $\xi$ requires only one iteration of the deflation. However, it is an embedded zero on the one-dimensional variety $x=0$. More results for the deflation on embedded singularities can be found in \cite{leykin2008numerical,Hauenstein2013}.
	\item 	\label{exm:1} The converse of Corollary \ref{neccesary} is not true in general. The system
	\begin{equation*}
		f(x,y,z) = \begin{bmatrix}
			x^2\\
			z^3 + xy\\
			y^2
		\end{bmatrix}
	\end{equation*}
	has an isolated singular zero $\xi$ at the origin
	with $\kappa =3,\rho=5$ and $\mu =12$.
	Note that
	$\mathcal{D}_{f,\xi}^{(2)}$ is isomorphic to the kernel of
	$C_{2}=[0_{3\times1}~0_{3\times3}~e_1~e_2~0_{3\times1}~e_3~0_{3\times1}~0_{3\times1}]$ since  $$\mathcal{C}_{f,\xi}^{(2)}=\spanz_\mathbb{C}\left\{1,d_{1},d_2,d_3,d_1^2,d_1d_2,d_1d_3,d_2^2,d_2d_3,d_3^2\right\}.$$ It means that $\rank\,C_{2} = \dim\,\mathcal{C}^{(2)}_{\f,\xi}-\dim\,\mathcal{D}^{(2)}_{\f,\xi}=3$. However, we have
	\[\opB(\xi)
=\left[
	\begin{array}{ccc}
		\lambda_1 & 0 & 0 \\
		\lambda_2 & \lambda_1 & 0\\
		0 & \lambda_2 & 0
	\end{array}
	\right],\]
	which implies that $\det \opB(\xi) \equiv0$. Consequently, the origin is not a deflation-one singular zero of $f(x,y,z)=0$ even though $\dim\,\mathcal{C}^{(2)}_{\f,\xi}-\dim\,\mathcal{D}^{(2)}_{\f,\xi}=3$. It requires two steps of deflation to regularize the zero.
 \item \label{running_example_2} It has been shown in \cite[Corollary 3]{DLZ:2009} that for an isolated singular zero of an analytic system, the same multiplicity structure is presented as in the truncated polynomial system from its Taylor series up to a certain order at the same zero. Hence, we may extend the concept of deflation-one singularity to analytic systems. Consider the system from \cite[Example 3]{DLZ:2009}
 	\begin{equation*}
 		f(x,y,z) = \begin{bmatrix}
 			{x}^{3}+z\sin \left( y \right)\\
 			{y}^{3}+x\sin \left( z \right)\\
 			{z}^{3}+y\sin \left( x \right)
 		\end{bmatrix}
 	\end{equation*}
 	has an isolated singular zero $\xi$ at the origin with $\kappa =3,\rho=4$ and $\mu =11$. 
 	We have
 	\[\opB(\xi)
 	=\left[
 	\begin{array}{ccc}
 		0 & \lambda_3 & \lambda_2 \\
 		\lambda_3 & 0 & \lambda_1\\
 		\lambda_2 & \lambda_1 & 0
 	\end{array}
 	\right]\]
 	which concludes that $\det \opB(\xi) = 2\lambda_1\lambda_2\lambda_3\not\equiv0$. Thus, $\xi$ is a deflation-one singular zero of $f(x,y,z)=0$.
 \end{enumerate}	
\end{exm}

\section{Two-step Newton's method}\label{sect4}

We present a two-step Newton's method for refining deflation-one singular zeros, and prove its quadratic convergence when an approximation $x$ is sufficiently close to the exact solution $\xi$.
To establish the quadratic convergence of the two-step Newton's method, we need to assume that $x$ satisfies the following assumption:

\begin{ass}\label{ass}
For a given polynomial system $f\in\mathbb{C}[X_1, \ldots, X_n]^n$ with a deflation-one singular zero $\xi\in\mathbb{C}^n$ and $\brx=\dim\ker Df(\exact)$, assume that a point $x\in\mathbb{C}^n$ satisfies that $\epsilon=\|x-\xi\|$, where $0<\epsilon<1$ is sufficiently small. 
\end{ass}
A concrete range for the value of $\epsilon$ can be obtained by convergence analysis for Newton's method (see \cite[Chapter 8]{Blum:1997}). For isolated singular solutions, it was done for multiplicity $2$ case \cite{Dedieu2001on} and for $\kappa=1$ case \cite{HJLZ2020}. We do not pursue the convergence analysis for deflation-one singularity in this paper.

\subsection{Perturbations}\label{sec:assumptions}

To analyze the convergence of Newton's method, we recall some known results in \cite{weyl,wedin}. 
Let $A\in\CC^{n\times n}$ be an $n\times n$ matrix with its singular value decomposition $A=U\cdot\Sigma \cdot V^*$ where $\Sigma=\diag(\sigma_1,\ldots,\sigma_n)$ with $\sigma_1\geq\cdots\geq\sigma_n$. Consider a perturbation $\tilde{A}=A+E$ of $A$ and its singular value decomposition $\tilde{A}=\tilde{U}\cdot \tilde{\Sigma} \cdot\tilde{V}^*$ where $\tilde{\Sigma}=\diag(\tilde{\sigma}_1,\ldots,\tilde{\sigma}_n)$ with $\tilde{\sigma}_1\geq\cdots\geq\tilde{\sigma}_n$. Weyl gave the error bounds for the paring singular values in \cite{weyl}:
\begin{equation}\label{sigma}
  |\tilde{\sigma}_i-\sigma_i|\leq\|E\|,~i=1,\ldots,n.
\end{equation}
Note that the norm used here is $2$-norm. For the rest of the paper, the norm notation denotes either $2$-norm or F-norm without distinction since they are equivalent to each other and are unitarily invariant.

Since singular vectors corresponding to a cluster of singular
values may be unstable, error comparisons are made instead for singular subspaces corresponding to the cluster of singular
values. Let
\[A=[U_1,U_2]\cdot\left[
               \begin{array}{cc}
                 \Sigma_1 & 0 \\
                 0 & \Sigma_2 \\
               \end{array}
             \right]\cdot\left[
               \begin{array}{c}
                 V_1^*  \\
                 V_2^* \\
               \end{array}
             \right]\mbox{ and }\tilde{A}=[\tilde{U}_1,\tilde{U}_2]\cdot\left[
               \begin{array}{cc}
                 \tilde{\Sigma}_1 & 0 \\
                 0 & \tilde{\Sigma}_2 \\
               \end{array}
             \right]\cdot\left[
               \begin{array}{c}
                 \tilde{V}_1^*  \\
                 \tilde{V}_2^* \\
               \end{array}
             \right],
\]
where $U_2,V_2,\tilde{U}_2,\tilde{V}_2\in\CC^{n\times\brx}$ and $\Sigma_2,\tilde{\Sigma}_2\in\CC^{\brx\times\brx}$. Define the canonical angles between $\im V_2$ and $\im \tilde{V}_2$ by
$\theta_j=\arccos\gamma_j$ for $j=1,\ldots,\brx$,
where $\gamma_1\geq\cdots\geq\gamma_{\brx}$ are the singular values of $\tilde{V}_2^*\cdot V_2$.
The corresponding singular subspaces $\im V_2$ and $\im \tilde{V}_2$ are said to be \textit{near} if the largest canonical angle $\theta_1$ is small. It is known from \cite{wedin} that if there is a $\delta>0$ such that
\begin{equation}\label{seper}
\min_{\substack{i=1,\ldots,n-\brx,\\j=1,\ldots,\brx}}|\tilde{\sigma}_i-\sigma_{n-\brx+j}|\geq\delta\mbox{ and }\min_{i=1,\ldots,n-\brx}\tilde{\sigma}_i\geq\delta,
\end{equation}
then
\begin{equation}\label{uv}
  \sin\theta_j\leq\frac{\|E\|}{\delta},~j=1,\ldots,\brx,
\end{equation}
so that the angles can be used as the error bounds.
The error bounds (\ref{uv}) are also applicable to the singular values of $\tilde{U}_2^*\cdot U_2$. Additionally, as $[V_1,V_2]$ and $[\tilde{V}_1,\tilde{V}_2]$ are both orthogonal, it can be concluded that the same error bounds hold for $\tilde{V}_1^*\cdot V_1$, and likewise for $\tilde{U}_1^*\cdot U_1$.

Let us consider an exact isolated zero $\xi$ and its approximation $x$, and  treat $Df(x)$ as an approximation of $Df(\xi)$. 
Let $\brx$ be the corank of  $Df(\exact)$,  
the singular value decompositions of $Df(\x)$ and $Df(\exact)$ are 
\begin{equation}\label{svd2}
Df(\x)=[U_1,U_2]\cdot \left[
               \begin{array}{cc}
                 \Sigma_1 & 0 \\
                 0 & \Sigma_2 \\
               \end{array}
             \right]\cdot\left[
               \begin{array}{c}
                 V_1^*  \\
                 V_2^* \\
               \end{array}
             \right]\mbox{ and }Df(\exact)=[\tilde{U}_1,\tilde{U}_2]\cdot \left[
               \begin{array}{cc}
                 \tilde{\Sigma}_1 & 0 \\
                 0 & 0 \\
               \end{array}
             \right]\cdot\left[
               \begin{array}{c}
                 \tilde{V}_1^*  \\
                 \tilde{V}_2^* \\
               \end{array}
             \right],
\end{equation}
where $U_1, V_1, \tilde{U}_1, \tilde{V}_1 \in \CC^{n \times (n-\brx)}$, $U_2, V_2, \tilde{U}_2, \tilde{V}_2 \in \CC^{n \times \brx}$, $\Sigma_1$ and $\tilde{\Sigma}_1$ are diagonal matrices of singular values $\sigma_1,\ldots,\sigma_{n-\brx}$ and $\tilde{\sigma}_1,\ldots,\tilde{\sigma}_{n-\brx}$, respectively, and $\Sigma_2$ is a diagonal matrix of singular values $\sigma_{n-\brx+1},\ldots,\sigma_n$. 
The following lemma estimates the errors between singular value decompositions of $Df(\x)$ and $Df(\exact)$ given in (\ref{svd2}). 
\begin{lem}\label{app}
Under Assumption \ref{ass}, let $\epsilon=\|x-\xi\|$ and $0<\epsilon<1$ satisfy that 
$$\|Df(\exact)-Df(x)\|<\frac{c}{4},$$
where $c=\tilde{\sigma}_{n-\brx} $ is  the smallest positive singular values of $Df(\xi)$. 
Then the following results hold:
\begin{enumerate}
  \item\label{lem:app1}  $\sigma_{n-\brx} \geq 3c/4$  and $\|\Sigma_2\|=\bigO(\epsilon)$,
  \item  $\tilde{V}_1^*\cdot V_1$ and $\tilde{V}_2^*\cdot V_2$ are well-conditioned, and so are $\tilde{U}_1^*\cdot U_1$ and $\tilde{U}_2^*\cdot U_2$. That is, condition numbers of these matrices are less than some constant $\mu$ greater than $1$. In addition,
  $\|\tilde{V}_1^*\cdot V_2\|=\bigO(\epsilon)$, $\|\tilde{V}_2^*\cdot V_1\|=\bigO(\epsilon)$, $\|\tilde{U}_1^*\cdot U_2\|=\bigO(\epsilon)$ and $\|\tilde{U}_2^*\cdot U_1\|=\bigO(\epsilon)$,
  \item\label{lem:app3}  Let $v\in \im V_2$ be unitary, then $\opB(x)=U_2^*\cdot D^2\f(\x)\cdot v\cdot V_2$ is well-conditioned.
\end{enumerate}
\end{lem}
\begin{proof}
Let $E=Df(\exact)-Df(x)$, then $\|E\|=\bigO(\epsilon)<c/4$ since $Df:\CC^n\rightarrow\CC^{n\times n}$ is Lipschitz continuous.
\begin{enumerate}
  \item According to the error bound (\ref{sigma}), we have
$\sigma_{n-\brx} \geq\tilde{\sigma}_{n-\brx}-\|E\| > 3c/4$ and 
$\sigma_{n-\kappa+j} \leq \tilde{\sigma}_{n-\kappa+j}+\|E\|= \bigO(\epsilon)<c/4 $ for $j=1,\ldots,\brx$.
Hence we have  $\|\Sigma_2\|=\bigO(\epsilon)$.

  \item Set $\delta=cp$ for some $1/4<p<3/4$. Then, for all $i=1,\dots, n-\kappa$ and $j=1,\dots, \brx$, we have $\tilde{\sigma}_i-\sigma_{n-\kappa+j}\geq\tilde{\sigma}_{n-\brx}-\|E\| > 3c/4>\delta$ and $\tilde{\sigma}_i\geq c>\delta$  so that the conditions in (\ref{seper}) are satisfied. Thus, according to the error bounds (\ref{uv}), we know that $\sin \theta_j\leq\|E\|/\delta<\frac{1}{4p}$ for $j=1,\ldots,\brx$ so that singular values $\gamma_j$ for $\tilde{V}_2^*\cdot V_2$ are greater than $\cos\theta_j>\frac{\sqrt{16p^2-1}}{4p}>0$ when $1/4<p<3/4$. It concludes that $\tilde{V}_2^*\cdot V_2$ is well-conditioned since its condition number is upper bounded by $\frac{4p}{\sqrt{16p^2-1}}$ which is greater than $1$ if $1/4<p<3/4$. The first part is proved by setting $\mu=\frac{4p}{\sqrt{16p^2-1}}$ for some $p$ in between $1/4$ and $3/4$. It is preferred to select a value near $3/4$ to achieve a $\mu$ value close to $1$.
  Likewise, $\tilde{V}_1^*\cdot V_1$, $\tilde{U}_1^*\cdot U_1$ and $\tilde{U}_2^*\cdot U_2$ are well-conditioned as well. 
  
  For the second part, since $\|U^*\cdot E\cdot \tilde{V}\|=\|E\|=\bigO(\epsilon)$ and 
  \begin{align*}
       U^*\cdot E\cdot \tilde{V}     =& U^*\cdot \left(Df(\exact)-Df(x)\right)\cdot\tilde{V}\\
       = &\begin{bmatrix}
                 U_1^*  \\
                 U_2^* \end{bmatrix}
                 \cdot\left(\tilde{U}\cdot\tilde{\Sigma}\cdot\tilde{V}^*-U\cdot \Sigma\cdot V^*\right)\cdot[\tilde{V}_1,\tilde{V}_2]\\
             = &
               \begin{bmatrix}
                 U_1^*  \\
                 U_2^* \\
               \end{bmatrix}\cdot\left(\tilde{U}_1\cdot\tilde{\Sigma}_1\cdot\tilde{V}_1^*-U_1\cdot \Sigma_1\cdot V_1^*-U_2\cdot \Sigma_2\cdot V_2^*\right)\cdot[\tilde{V}_1,\tilde{V}_2]\\
             = &\begin{bmatrix}
U_1^*\cdot\tilde{U}_1\cdot\tilde{\Sigma}_1-\Sigma_1\cdot V_1^*\cdot \tilde{V}_1 & -\Sigma_1\cdot V_1^*\cdot \tilde{V}_2 \\
                 U_2^*\cdot \tilde{U}_1\cdot \tilde{\Sigma}_1 & -\Sigma_2\cdot V_2^*\cdot\tilde{V}_2
             \end{bmatrix},               \end{align*}
  we have $\|\tilde{V}_2^*\cdot V_1\|=\bigO(\epsilon)$ and $\|\tilde{U}_1^*\cdot U_2\|=\bigO(\epsilon)$. Likewise, $\|\tilde{V}_1^*\cdot V_2\|=\bigO(\epsilon)$ and $\|\tilde{U}_2^*\cdot U_1\|=\bigO(\epsilon)$ hold.
  \item  Let $v=V_2\cdot \boldsymbol{\lambda}_2\in\im V_2$ be unitary. 
Then, we have
\begin{align*}
  \opB(x) & =U_2^*\cdot D^2\f(x)\cdot V_2\cdot \boldsymbol{\lambda}_2\cdot V_2 \\
        & =U_2^*\cdot D^2\f(\exact)\cdot V_2\cdot \boldsymbol{\lambda}_2\cdot V_2 + U_2^*\cdot \left(D^2\f(x)-D^2\f(\exact)\right)\cdot V_2\cdot \boldsymbol{\lambda}_2\cdot V_2 \\
        & =U_2^*\cdot[\tilde{U}_1,\tilde{U}_2]\cdot\left[
               \begin{array}{c}
                 \tilde{U}_1^*  \\
                 \tilde{U}_2^* \\
               \end{array}
             \right]\cdot D^2\f(\exact)\cdot V_2\cdot \boldsymbol{\lambda}_2\cdot V_2 + \bigO(\epsilon)\\
   & =U_2^*\cdot\tilde{U}_2\cdot\tilde{U}_2^*\cdot D^2\f(\exact)\cdot V_2\cdot \boldsymbol{\lambda}_2\cdot V_2\\
   &\quad\quad\quad+U_2^*\cdot\tilde{U}_1\cdot\tilde{U}_1^*\cdot D^2\f(\exact)\cdot V_2\cdot \boldsymbol{\lambda}_2\cdot V_2 + \bigO(\epsilon)\\
   & =U_2^*\cdot\tilde{U}_2\cdot\tilde{U}_2^*\cdot D^2\f(\exact)\cdot V_2\cdot \boldsymbol{\lambda}_2\cdot V_2 + \bigO(\epsilon)\\
   & = U_2^*\cdot\tilde{U}_2\cdot\tilde{U}_2^*\cdot D^2\f(\exact)\cdot (\tilde{V}_2\cdot\tilde{V}_2^*)\cdot V_2\cdot \boldsymbol{\lambda}_2\cdot (\tilde{V}_2\cdot\tilde{V}_2^*)\cdot V_2 + \bigO(\epsilon)\\
   & = (U_2^*\cdot\tilde{U}_2)\cdot\opB(\xi)\cdot(\tilde{V}_2^*\cdot V_2) + \bigO(\epsilon)
 \end{align*}
Observe that the third equality can be derived by the fact that $\|D^2f(x)-D^2f(\xi)\|=\bigO(\epsilon)$, the fifth equality is obtained by $\|U_2^*\cdot \tilde{U}_1\|=\bigO(\epsilon)$, and the last equality follows since $(\tilde{V}_2\cdot\tilde{V}_2^*)\cdot V_2\cdot \boldsymbol{\lambda}_2$ is a unitary vector in $\im \tilde{V}_2$. 
Note that $\opB(\xi)$ is invertible by Theorem \ref{opB}. Because $U_2^*\cdot\tilde{U}_2$ and $\tilde{V}_2^*\cdot V_2$ are well-conditioned, by the error bounds (\ref{sigma}), it concludes that $\opB(x)$ is well-conditioned as well. \qedhere
\end{enumerate}
\end{proof}

Combining all these, we present the two-step Newton's method algorithm like the following:

    \begin{algorithm}[H]
	\caption{Two-step Newton's method for refining deflation-one singular zeros}
	\label{revisedNewton}
	\begin{algorithmic}[1]\label{alg1}
		
		\REQUIRE~~\\
		$\f$: a square polynomial system from $\CC[X_1,\ldots,X_n]$; \\
		$\x$:  an approximation of a deflation-one singular zero $\exact$; \\
		$\tau$: a tolerance for determining the numerical rank of $\Jac$; \\
		\ENSURE~~\\
		$x''$:  the refined approximation of $\xi$;\\
		
		\STATE calculate $\Jac=U\cdot\Sigma\cdot V^{\ast}$ via SVD, \\
		where $U=[u_1,\ldots,u_n]$, $V=[v_1,\ldots,v_n]$ are unitary and $\Sigma=\diag(\sigma_1,\ldots,\sigma_n)$ with $\sigma_1\geq\cdots\geq\sigma_n\geq0$.
		\STATE determine $n-\brx=\arg\max\{i\mid\sigma_i>\tau\}$ and denote $U=[U_1,U_2]$, $V=[V_1,V_2]$, $\Sigma_1=\diag(\sigma_1,\dots, \sigma_{n-\kappa})$, $\Sigma_2=\diag(\sigma_{n-\brx+1},\ldots,\sigma_n)$
		where $U_2=[u_{n-\brx+1},\ldots,u_{n}]$, $V_2=[v_{n-\brx+1},\ldots,v_{n}]$.\\
		\STATE refine $\x$ to
			$\x'=\x-V_1\cdot\Sigma_1^{-1}\cdot U_1^*\cdot\f(\x)$.
		\STATE construct the invertible $\brx\times\brx$ matrix $\opB'=U_2^*\cdot D^2\f(\x')\cdot v\cdot V_2$
		where $v\in\im V_2$ is a random sample with $\|v\|=1$.
		\STATE solve the linear equations $\opB'\cdot\boldsymbol{\delta}
			=-U_2^*\cdot D\f(\x')\cdot v$.
		\STATE refine $\x'$ to $\x''=\x' + V_2 \cdot \boldsymbol{\delta}$.
	\end{algorithmic}
\end{algorithm}

\subsection{Quadratic convergence}

We prove the quadratic convergence of Algorithm \ref{revisedNewton} under Assumption \ref{ass}.

\begin{thm}\label{mainresult}
Suppose that a square polynomial system $f\in \mathbb{C}[X_1,\dots, X_n]^n$ with its deflation-one singular zero $\xi$, and an approximation $x$ are given under Assumption \ref{ass}. 
Then, applying the two-step Newton's method proposed in Algorithm \ref{revisedNewton} iteratively, an approximation $x$ converges quadratically to the exact zero $\xi$.
\end{thm}

Theorem \ref{mainresult} is deduced  immediately from 
 Lemma \ref{lem:1strefinement} and Lemma \ref{lemma14},  which measure the error at each refinement step of Algorithm \ref{revisedNewton}.

\begin{lem}\label{lem:1strefinement}
Let $y=-V_1\cdot\Sigma_1^{-1}\cdot U_1^*\cdot \f(\x)$. Then, the projection $x'=x+y$ satisfies $\left\|\x'-\exact-V_2\cdot\boldsymbol{\alpha}\right\|=\bigO(\epsilon^2)$
where $\boldsymbol{\alpha}=V_2^*(\x'-\exact)$.
\end{lem}
\begin{proof}
	Note that
	\[\left\|\x'-\exact-V_2\cdot\boldsymbol{\alpha}\right\|=\left\|V^*\left(\x'-\exact-V_2\cdot\boldsymbol{\alpha}\right)\right\|=\|V_1^*(\x'-\exact)\|=\|\tilde{\Sigma}_1^{-1}\cdot\tilde{U}_1^*\cdot D\f(\exact) (\x'-\exact)\|.\]
	Hence, we claim that $\| Df(\xi)(x'-\xi)\|=\bigO(\epsilon^2)$. Considering Taylor expansion of $f(x')$ at $\xi$, we have
	\[
	\f(\x') = \f(\exact)+D\f(\exact) (\x'-\exact) + \sum_{k\geq2}\frac{D^k\f(\exact)}{k!}(\x'-\exact)^k.
	\]
	Since $f(\xi)=0$, we have $\|Df(\xi)(x'-\xi)\|\leq \| f(x')\|+\bigO(\epsilon^2)$. Thus, it is enough to show that $\|f(x')\|=\bigO(\epsilon^2)$. By Taylor expansion of $f(x')$ at $x$, we have
	\[\f(\x')=\f(\x)+\Jac y+
	\sum_{k\geq2}\frac{D^k\f(\x)}{k!}y^k.\]
Since $\|y\|=\bigO(\epsilon)$, we know that the last term is $\bigO(\epsilon^2)$.
Additionally, because $U$ is a unitary matrix, we get that $\|f(x')\|\leq \|f(x)+Df(x)y\|+\bigO(\epsilon^2)=\|U^*\cdot f(x)+U^*\cdot Df(x)y\|+\bigO(\epsilon^2)$. Furthermore,
\begin{align*}
\|U^*\cdot f(x)+U^*\cdot Df(x)y\| & =\|U_1^*\cdot f(x)+U_1^*\cdot Df(x)y\|+\|U_2^*\cdot f(x)+U_2^*\cdot Df(x)y\| \\
&=\|U_1^*\cdot f(x)-\Sigma_1\cdot V_1^*\cdot V_1\cdot\Sigma_1^{-1}\cdot U_1^*\cdot f(x)\|\\
&\qquad\qquad\qquad+\|U_2^*\cdot f(x)-\Sigma_2\cdot V_2^*\cdot V_1\cdot\Sigma_1^{-1}\cdot U_1^*\cdot f(x)\| \\
&=\|U_1^*\cdot f(x)-U_1^*\cdot f(x)\|+\|U_2^*\cdot f(x)-0\|\\
&=\|U_2^*\cdot f(x)\|.
\end{align*}
Finally, by Taylor's expansion of $f(\xi)$ at $x$, we get
\[0=\f(\exact) = \f(\x)+\Jac (\exact-\x) + \sum_{k\geq2}\frac{D^k\f(\x)}{k!}(\exact-\x)^k,\]
and hence $\|U_2^*\cdot f(x)\|\leq \|U_2^*\cdot Df(x)(\xi-x)\|+\bigO(\epsilon^2)\leq \|\Sigma_2\cdot V_2^*(\xi-x)\|+\bigO(\epsilon^2)=\bigO(\epsilon^2)$.
The last inequality is obtained by Lemma \ref{app} because $\|\Sigma_2\|=\bigO(\epsilon)$. Thus, $\|f(x')\|=\bigO(\epsilon^2)$, and the result follows.
 \end{proof}

Note that $x'$ may not necessarily converge closer to $\exact$ than $x$. However, it has been adjusted to a special location such
that $x'-V_2\cdot\boldsymbol{\alpha}$ allows the quadratic convergence to $\xi$ 
if the step-length vector $\boldsymbol{\alpha}$ is properly estimated. The following lemma shows  that the solution $\boldsymbol{\delta}$ of the linear equation $\opB'\cdot\boldsymbol{\delta}
=-U_2^*\cdot D\f(\x')\cdot v$ satisfies that $\|-\boldsymbol{\alpha}-\boldsymbol{\delta}\|= \bigO(\epsilon^2)$.

\begin{lem}\label{lemma14}
Suppose that $\boldsymbol{\delta}
        =-\opB'^{-1}\cdot U_2^*\cdot D\f(\x')\cdot v$
where $\opB'=U_2^*\cdot D^2\f(\x')\cdot v\cdot V_2$ and $v\in \im V_2$ is unitary. Then, the refinement $x''=x'+V_2\cdot \boldsymbol{\delta}$ satisfies that $\|\xi-x''\|=\bigO(\epsilon^2)$.
\end{lem}
\begin{proof}
According to Lemma \ref{app}, we may consider that $\opB'$ is invertible. Therefore,
\begin{align*}
  \left\|\boldsymbol{\delta}\right\| & =\left\|\opB'^{-1}\cdot U_2^*\cdot Df(x')\cdot v\right\| \\
        & =\left\|\opB'^{-1}\cdot U_2^*\cdot(\Jac+\bigO(\|x'-x\|))\cdot v\right\| \\
        & \leq\left\|\opB'^{-1}\cdot U_2^*\cdot\Jac\cdot v\right\|+\bigO(\epsilon)\\
   & =\bigO(\epsilon). 
\end{align*}
The inequality is obtained by the fact that $\|x'-x\|=\|y\|=\bigO(\epsilon)$.
Considering Taylor's expansion of $Df(x'')\cdot v$ at $x'$, we have
	\[Df(x'')\cdot v=Df(x')\cdot v+D^2f(x')\cdot v(x''-x')+\sum_{k\geq2}\frac{D^{k+1}f(x')\cdot v}{k!}(x''-x')^k.\]
We claim that $\|Df(x'')\cdot v\|=\bigO(\epsilon^2)$.
From the fact that $$U_2^*\cdot D^2f(x')\cdot v(x''-x')=(U_2^*\cdot D^2f(x')\cdot v\cdot V_2)\cdot\boldsymbol{\delta}=\opB'\cdot \boldsymbol{\delta}$$ and $\opB'\cdot\boldsymbol{\delta}
=-U_2^*\cdot D\f(\x')\cdot v$,
it can be derived that 
	\[Df(x'')\cdot v=\sum_{k\geq2}\frac{D^{k+1}f(x')\cdot v}{k!}(x''-x')^k.\]
Therefore, we have $\|U_2^*\cdot Df(x'')\cdot v\|=\|Df(x'')\cdot v\|=\bigO(\epsilon^2)$.
By Taylor expansion of $Df(\xi)\cdot v$ at $x''$, we know that 
	\begin{equation}\label{eq:taylorOfxi}
		Df(\xi)\cdot v=Df(x'')\cdot v+D^2f(x'')\cdot v(\xi-x'')+\sum_{k\geq2}\frac{D^{k+1}f(x'')\cdot v}{k!}(\xi-x'')^k.
	\end{equation}
We consider the second term from the expansion. By Lemma \ref{lem:1strefinement}, we have that $$U_2^*\cdot D^2f(x'')\cdot v(\xi-x'')=U_2^*\cdot D^2f(x'')\cdot v(\xi-x'+x'-x'')=U_2^*\cdot D^2f(x'')\cdot v(V_2\cdot(-\boldsymbol{\alpha}-\boldsymbol{\delta})).$$ Also, using Lemma \ref{app}, we can show that $U_2^*\cdot Df(\exact)\cdot v=(U_2^*\cdot \tilde{U}_1)\cdot\tilde{\Sigma}_1\cdot(\tilde{V}_1^*\cdot V_2)\cdot\boldsymbol{\lambda}_2=\bigO(\epsilon^2)$. Since $\|U_2^*\cdot Df(x'')\cdot v\|=\bigO(\epsilon^2)$ is already derived, we have $\|U_2^*\cdot D^2f(x'')\cdot v(\xi-x'')\|=\bigO(\epsilon^2)$ from the expansion (\ref{eq:taylorOfxi}). It means that $\|U_2^*\cdot D^2f(x'')\cdot v(V_2\cdot(-\boldsymbol{\alpha}-\boldsymbol{\delta}))\| =\bigO(\epsilon^2)$. Because $\|\xi-x''\|=\bigO(\epsilon)$, applying the similar argument of the proof for Lemma 11 (\ref{lem:app3})  shows that $U_2^*\cdot D^2f(x'')\cdot v\cdot V_2$ is invertible. Thus, $\|-\boldsymbol{\alpha}-\boldsymbol{\delta}\|=\bigO(\epsilon^2)$, and so $$\|\xi-x''\|=\|\xi-x'+x'-x''\|=\|V_2\cdot(-\boldsymbol{\alpha}-\boldsymbol{\delta})\|=\bigO(\epsilon^2).$$
The result follows.
\end{proof}

\begin{rmk} 

\begin{enumerate}
    \item A correct estimation of $\brx$ is necessary for the quadratic convergence of Algorithm \ref{revisedNewton}, which can be achieved by a suitable choice of a tolerance $\tau$. 
Assume $\epsilon=\|x-\xi\|$ is sufficiently small such that 
$\|Df(\exact)-Df(x)\|<c/4$ and $\|\opB(\xi)-\opB(x)\|<c'/4$ where $c$ and $c'$ are the smallest positive singular values of $Df(\xi)$ and $\opB(\xi)$ respectively. Then setting $\tau=\min\{c/2,c'/2\}$ guarantees the estimation of $\brx$ determined by $n-\brx=\arg\max\{i\mid\sigma_i>\tau\}$ is correct. Even though $c$ and $c'$ are not known a priori, $\tau$  can still be set suitably from the gap between singular values of $Df(x)$. 
See \cite{Kung} for more results about determining numerical ranks of perturbed matrices.
 \item When $\brx=n$, the first refinement in Algorithm \ref{revisedNewton} is redundant since $\left\|\x'-\exact-V\cdot\boldsymbol{\alpha}\right\|=0$  holds
for  $\boldsymbol{\alpha}=V^*(\x'-\exact)$. Therefore, if $\brx=n$, we will skip the first three steps. The quadratic convergence of Algorithm \ref{revisedNewton} still holds in this case.   See Example \ref{ex18}. 
\end{enumerate}
\end{rmk}

 \subsection{Examples}
In this subsection, we provide examples illustrating how the two-step Newton's method works. Each example demonstrates possible scenarios that can happen while the algorithm runs. The first example shows how the algorithm works regularly.

\begin{exm}[Example \ref{running_example_1} continued]
\label{running_example_1.1}
  Let $x=[1.001,0.999,1.001]^\top$ be an approximation of the deflation-one singular zero $\xi=[1,1,1]^\top$ with an error $\|x-\xi\|\approx1.7\times10^{-3}$ and $\tau=0.1$ be a tolerance to identify the type of singularity.

  The singular value decomposition of $Df(x)$ gives 
  \begin{equation*}
          Df(x) = \begin{bmatrix}
               1.002 & 1 & 1 \\
               1 & 0.998 & 1 \\
               1 & 1 & 1.002
\end{bmatrix}=U\cdot \Sigma\cdot V^*
  \end{equation*}
  where $\Sigma=\diag(3.0007,0.0020,0.0007)$, $$U=\begin{bmatrix} -0.5776 & 0.7071 & 0.4079 \\
               -0.5768 & 3.6\times10^{-15} & -0.8169 \\
               -0.5776 & -0.7071 & 0.4079\end{bmatrix}, \text{ and }~ V^*=\begin{bmatrix}
                                  -0.5776 & -0.5768 & -0.5776  \\
               0.7071 & 4.5\times10^{-14} & -0.7071 \\
               -0.4079 & 0.8169 & -0.4079
               \end{bmatrix}.$$ It determines that $\kappa=2$, and hence we know that $U=[U_1,U_2]$ and $V=[V_1,V_2]$ with  $U_1,V_1\in \mathbb{C}^{3\times 1}$ and $U_2,V_2\in\mathbb{C}^{3\times 2}$.
As the first refinement, we update $x$ to $$x'=x-\frac{1}{3.0007}V_1\cdot U_1^*\cdot f(x)=[  1.000666,
               0.998667,
               1.000666]^\top.$$
                 with an error $\|x'-\xi\|\approx1.6\times10^{-3}$. Choosing $v=[2,-1,-1]^\top$ from the numerically approximated kernel, we have invertible $\opB'=\begin{bmatrix}
                       1.0000 & -1.7305  \\
               1.7305 &  1.0018
                 \end{bmatrix}$. Then, we have the step length $\boldsymbol{\delta}=-\opB'\cdot U_2^*\cdot Df(x')\cdot v=[-7.1\cdot 10^{-7}, 0.0016]^\top$. Finally, we update $x'$ to 
                 $$x''=x'+V_2\cdot \boldsymbol{\delta}=[               0.99999967,
               1.00000067,
               1.00000067]^\top
$$
  with an error $\|x''-\xi\|\approx1.0\times10^{-6}$.
  Therefore, the two-step Newton's method achieves quadratic convergence. The pictorial description is in Figure \ref{sketch}.
\end{exm}

\begin{figure}
\centering

\tdplotsetmaincoords{87}{218}
\begin{tikzpicture}[tdplot_main_coords,scale=2]
\draw[line width=0.1mm,->] (0.5,0.5,.5) -- (3,0.5,0.5) node[anchor=north east]{$x$};
\draw[dashed,line width=0.1mm] (0.5,0.5,1.5) -- (0.5,0.5,2.05);
    \def\x{.5}
    \draw[thin] (-1.4,1.8,2.6) -- (1.6,-1.2,2.6);
    \draw[thin] (3,0,0) -- (1.6,-1.2,2.6) ;
    \draw[thin] (3,0,0) -- (0,3,0) ;
    \draw[thin] (-1.4,1.8,2.6) -- (0,3,0) ;
    \filldraw[
        draw=orange,
        fill=orange!50,
        fill opacity=.7,
    ] (-1.4,1.8,2.6) 
    -- (1.6,-1.2,2.6) 
    -- (3,0,.0)
    -- (0,3,0) 
    -- cycle;

\draw[color=red] (1+1,1-1,1+1) node {$\boldsymbol{\bullet}$};
\draw (2,0,2) node[anchor=south east] {$x$};
\draw[color=blue] (1+.666,1-1.333,1.666) node {$\boldsymbol{\bullet}$};
\draw (1+.666,1-1.333,1.666) node[anchor=south west] {$x'$};
\draw[color=black] (.99967,1.00067,1.00067) node {$\boldsymbol{\bullet}$};
\draw (.99967,1.00067,1.00067) node[anchor=north] {$x''$};
\draw[color=black] (.5,.5,.5) node {$\boldsymbol{\bullet}$};
\draw[color=black] (.4,.4,.4) node[anchor=north] {$(0.9995,0.9995,0.9995)$};

\draw[line width=0.5mm,->] (1.99,-0.01,1.99) -- (1+.7,1-1.3,1.7);
\draw[line width=0.5mm,dashed,->] (1+.666,1-1.333,1.666) -- (.95,1.05,1.05);

\draw[line width=0.1mm,->] (0.5,0.5,0.5) -- (0.5,3.2,0.5) node[anchor=north west]{$y$};
\draw[line width=0.1mm] (0.5,0.5,0.5) -- (0.5,0.5,1.5);
\draw[line width=0.1mm,->] (0.5,0.5,2.05) -- (0.5,0.5,3) node[anchor=south]{$z$};
\end{tikzpicture}
\hspace{5pt}
\begin{tikzpicture}[scale=18]
\draw[line width=0.1mm,->] (.9-.03,.96) -- (.9-.03,1.25) node[anchor=north east]{$z$};
\draw[line width=0.1mm,->] (.9-.03,.96) -- (.65,.96) node[anchor=south]{$x,y$};    

\draw[thin] (.9,1.2) -- (.9-0.15,.9);    
\draw[thick,color=orange,opacity=.9] (.9,1.2) -- (.9-0.15,.9);    

\draw[color=red] (.8,1.1) node {$\boldsymbol{\bullet}$};
\draw[color=blue] (.9-.07,1.07) node {$\boldsymbol{\bullet}$};
\draw[color=black] (.9-.099967,1.000667) node {$\boldsymbol{\bullet}$};
\draw[color=black] (.8,1) node {$\boldsymbol{\bullet}$};
\draw[color=black] (.9-.03,.96) node {$\boldsymbol{\bullet}$};
\draw[color=black] (.9,.93) node {$(0.9993,0.9993,0.9996)$};

  \draw[dashed,line width=0.4mm,domain=.7:.9, smooth, variable=\x, red] plot ({\x}, {sqrt((.1)^2-(\x-0.8)^2)+1});
  \draw[dashed,line width=0.4mm,domain=.7:.9, smooth, variable=\x, red] plot ({\x}, {-sqrt((.1)^2-(\x-0.8)^2)+1});
  \draw[dashed,line width=0.4mm,domain=.8-.076:.8+0.076, smooth, variable=\x, blue] plot ({\x}, {sqrt((.076)^2-(\x-0.8)^2)+1});
  \draw[dashed,line width=0.4mm,domain=.8-.076:.8+0.076, smooth, variable=\x, blue] plot ({\x}, {-sqrt((.076)^2-(\x-0.8)^2)+1});

\draw[line width=0.3mm,color=red] (.8-.069,.92) -- (.8+0.076,1.21);    
\draw[line width=0.3mm,color=red] (.8-.011,.92) -- (.8+0.081,1.11);    
\draw[line width=0.3mm,color=blue] (.8-.031,.92) -- (.8+0.081,1.144);    
\draw[line width=0.3mm,color=blue] (.8-.049,.92) -- (.8+.081,1.18);

\draw[line width=0.5mm,->] (.8,1.1) -- (.9-.073,1.073);
\draw[line width=0.5mm,dashed,->] (.9-.07,1.07) -- (.9-.099,1.005);

\end{tikzpicture}
\caption{The description of Example \ref{running_example_1.1}. The point $x$ is depicted in red, while the blue point $x'$ and the black point $x''$ are obtained through the refinement steps.  The orange plane represents the kernel of the Jacobian of $f$.  The origin of the coordinate axes has been shifted to a suitable point due to scaling issues. The second figure shows the projection onto $y=x$ plane. The solid blue and red lines are contours of $\|f\|=10^{-3.5}$ and $10^{-3}$ respectively and dotted circles describe the distance from $\xi$ to $x$ and $x'$. The distance from $\xi$ to $x''$ is negligible compared to the other distances and thus omitted.}
\label{sketch}
\end{figure}
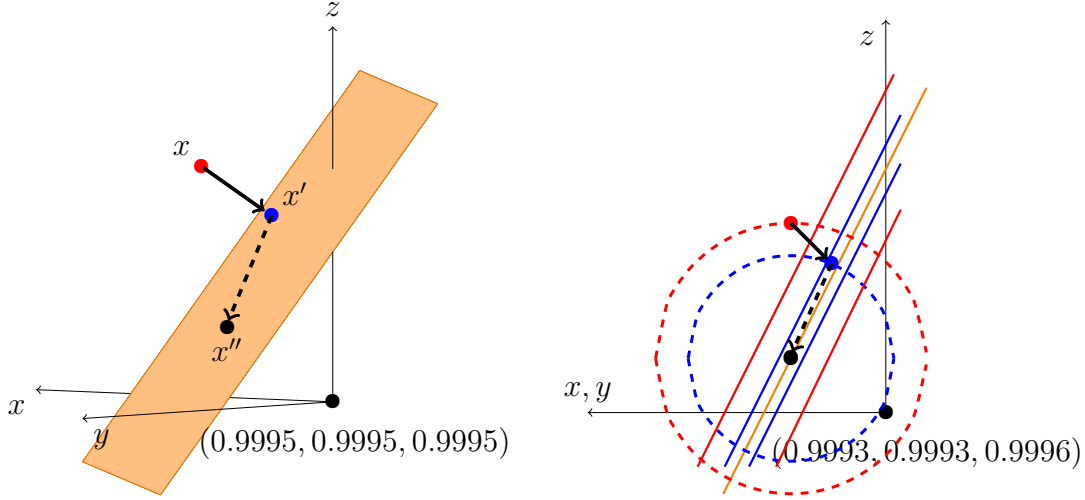

If $x-\xi$ is approximately orthogonal to  $\ker Df(x)$, then quadratic convergence can already be achieved after running the first refinement (that is, step 1, 2, and 3  in  Algorithm \ref{revisedNewton}). The following example describes this case.

\begin{exm}[Example \ref{running_example_1} continued]\label{ex17}
  Let $x=[1.001,1.001,1.001]^\top$ be an approximation of the deflation-one singular zero $\xi=[1,1,1]^\top$ with an error $\|x-\xi\|\approx1.7\times10^{-3}$. A tolerance is chosen by $\tau=0.1$.
  The singular value decomposition of $Df(x)$ produces $\Sigma=\diag(3.0020,0.0020,0.0020)$ with
  $$U=\begin{bmatrix}             -0.5774 & 0.8165 & -8.8\times10^{-17} \\
               -0.5774 & -0.4082 & -0.7071 \\
               -0.5774 & -0.4082 & 0.7071
  \end{bmatrix}, \text{ and }~ V^*=\begin{bmatrix}
                                     -0.5774 & -0.5774 & -0.5774  \\
               0.8165 & -0.4082 & -0.4082 \\
               0 & -0.7071 & 0.7071
               \end{bmatrix}$$
  so that we estimate $\kappa=2$. Note that $x-\xi=[0.001,0.001,0.001]^\top$ approximately spans $\{\ker Df(x)\}^\bot$. The first refinement returns 
$$x'=x-\frac{1}{3.0020}V_1\cdot U_1^*\cdot f(x)=[  1.00000033,
               1.00000033,
               1.00000033]^\top$$              
               with an error $\|x'-\xi\|\approx5.8\times10^{-7}$. Therefore, the first refinement achieves quadratic convergence.
\end{exm}

The last example shows that if $\brx=n$ holds, then the first three steps should be redundant, and the last three steps will achieve quadratic convergence.

\begin{exm}[Example 9 (\ref{running_example_2}) continued]\label{ex18}
Let $x=[0.0001,0.0001,0.0001]^\top$ be an approximation of the deflation-one singular zero $\xi$ at the origin with an error $\|x-\xi\|\approx1.7\times10^{-4}$ and $\tau=0.1$. Then, the SVD of $Df(x)$ returns $\Sigma=\diag(2\times 10^{-4},2\times 10^{-4},2\times 10^{-4})$, so that $\kappa=3$. Therefore, $x'=x$. Choosing $v=[2,-1,-1]^\top$, we have 
\[    \opB' = \begin{bmatrix}
               5.1\times10^{-17} & -2.3\times10^{-17} & -1.4134 \\
               8.2\times10^{-9} & 2.0006 & 1.4\times10^{-8} \\
               1.4134 & 3.6\times10^{-15} & -2.0006
             \end{bmatrix}
  \quad\text{and}\quad
    \boldsymbol{\delta} = \begin{bmatrix}
               1.7326\times10^{-4} \\
               -7.1\times10^{-13} \\
               0
             \end{bmatrix}
\]
and we get $$    x''=[
               -3.0019\times10^{-8},
               -3.0019\times10^{-8},
               -3.0018\times10^{-8}]^\top$$
with an error $\|x''-\xi\|\approx5.2\times10^{-8}$.
\end{exm}

\section{Experiments}\label{sect5}

We present extensive experimental results to demonstrate the effectiveness of our proposed two-step Newton's method. All experiments were conducted using Maple 2018 \cite{maple1994waterloo} on a MacBook Pro equipped with a 2.3 GHz 8-Core Intel Core i9 processor and 16 GB 2400 MHz DDR4 memory running Catalina. The Maple environment variable was set to Digits:=14. 

\subsection{Performance} Our primary goal in this section is to evaluate the effectiveness and stability of our new algorithm in refining deflation-one singular zeros.

\subsubsection{Effectiveness}

The effectiveness is measured by observing the distance $\|x-\xi\|$ from an approximation $x$ to the exact zero $\xi$ as the number of iterations increases. We use an initial guess with two digits of precision and iterate the algorithm 3 times. We apply Algorithm \ref{revisedNewton} to the benchmark examples presented below, and the results are reported in Table \ref{table:experimentalResults}.
The tolerance for each example is chosen independently to accurately estimate the breadth of the deflation-one singular zero, which is not included in the table. 
The result shows that for benchmark examples, the $\|x-\xi\|$ exponent decreases to at least $-10$ within $3$ iterations.

\begin{table}
\centering
\begin{tabular}{c|c|c|c|c|c}
    Systems & Zeros & $\kappa$ & $\rho$ & $\mu$ & Numerical Errors \\
    \hhline{=|=|=|=|=|=}
    cbms1 \cite{sturmfels2002solving} & $(0,0,0)$ & $3$ & $4$ & $11$ & $-02\rightarrow -03 \rightarrow -05\rightarrow -10$ \\
    cbms2 \cite{sturmfels2002solving} & $(0,0,0)$ & $3$ & $3$ & $8$ & $-02\rightarrow <-20 \rightarrow <-20 \rightarrow <-20$\\
    mth191 \cite{leykin2006newton} & $(0,1,0)$ & $2$ & $2$ & $4$ & $-02\rightarrow- 04 \rightarrow -09\rightarrow -17$\\
    KSS \cite{kobayashi1998numerical} & $(1,1,1,1,1)$ & $4$ & $4$ & $16$ & $-02\rightarrow -05 \rightarrow -12\rightarrow <-20$\\
    Caprasse \cite{moritsugu1999multiple} & $(2,-\sqrt{3}i,2,\sqrt{3}i)$ & $2$ & $2$ & $4$ & $-02\rightarrow -05 \rightarrow -10\rightarrow -13$\\
    Cyclic9 \cite{faugere1999new} & $^*C_9$ & $2$ & $2$ & $4$ & $-02\rightarrow -05 \rightarrow -10\rightarrow -12$
\end{tabular}
\caption{The result of Algorithm \ref{revisedNewton} on benchmark examples. The column ``Numerical Errors'' shows the exponent of $\|x-\xi\|$ at each iteration, and the symbol $\rightarrow$ means the change of exponents. The notation $<-20$ indicates that the exponent of $\|x-\xi\|$ goes less than $-20$.}
{\Small$^*C_9=(z_0,z_1,z_2,z_0,-z_2,-z_1,z_0,-z_2,-z_1)$ with $z_0\approx -.9396926-.3520201i,z_1\approx -2.4601472-.8954204i,z_2\approx -.3589306-.1306401i$.}
\label{table:experimentalResults}
\end{table}

\subsubsection{Stability}

For stability, we measure how an approximation converges to the exact zero with respect to the change of the initial precision, the distance from the other zero, and the specified tolerance. 
We apply Algorithm \ref{revisedNewton} on the system
\begin{equation}\label{exm:stability}
	f(x,y,z) = \begin{bmatrix}
		x^2\\
		y^2\\
		z^2+10^{-k}z
	\end{bmatrix}
\end{equation}
with two deflation-one singular zeros at $\xi=[0,0,0]^\top$ and $\eta=[0,0,-10^{-k}]^\top$. Both singular zeros have $\kappa =2,\rho=2$ and $\mu =4$. 

Firstly, we fix $\eta=[0,0,-10^{-2}]^\top$ and $\tau=10^{-2}$, and measure how $\|x-\xi\|$ changes depending on the precision of an approximation $x$. The result is given in Table \ref{table:stability1}. The result shows when a more precise approximation is provided, the convergence becomes faster.

\begin{table}
	\centering
	\begin{tabular}{c|c}
		$x$ &  Numerical Errors \\
		\hhline{=|=}
$(10^{-5},10^{-5},10^{-5})$ & $-04\rightarrow -08 \rightarrow -14\rightarrow -26$ \\
 $(10^{-4},10^{-4},10^{-4})$  & $-03\rightarrow -06 \rightarrow -10 \rightarrow -18$\\
$(10^{-3},10^{-3},10^{-3})$ & $-02\rightarrow- 04 \rightarrow -06\rightarrow -10$
	\end{tabular}
	\caption{The change of numerical errors depending on the precision of $x$ when Algorithm \ref{revisedNewton} is applied on the system (\ref{exm:stability}).}
	\label{table:stability1}
\end{table}

Secondly, we fix an approximation $x=[10^{-3},10^{-3},10^{-3}]^\top$, but vary the value of $k$ for $\eta$ and the tolerance $\tau$ to see how they affect the convergence. For each case, we estimate the breadth, denoted by $\kappa^*$. The result is summarized in Table \ref{table:stability2}.
When $\eta$ becomes closer to $\xi$ (the first and second row) or $\tau$ is set larger (the last row), two singular roots $\xi$ and $\eta$ are considered as a deflation-one singular zero at $(4\xi+4\eta)/8$ with $\kappa =3,\rho=3$ and $\mu =8$. So, $x$ converges to the point $(4\xi+4\eta)/8$.
\begin{table}
\centering
\begin{tabular}{c|c|c|c}
    $\eta$ & $\tau$ & $\brx^*$ & Numerical Errors \\
    \hhline{=|=|=|=}
    $(0,0,10^{-3})$ & $10^{-2}$ & $3$ & $-02\rightsquigarrow <-20 \rightsquigarrow <-20\rightsquigarrow <-20$\\
    $(0,0,10^{-4})$ & $10^{-2}$ & $3$ & $-02\rightsquigarrow -18 \rightsquigarrow <-20\rightsquigarrow <-20$\\
    $(0,0,10^{-2})$  & $10^{-1}$ & $3$ & $-02\rightsquigarrow -18 \rightsquigarrow <-20\rightsquigarrow <-20$
\end{tabular}
\caption{The convergence of Algorithm \ref{revisedNewton} on system (\ref{exm:stability}) depending on changes of $\eta$ and $\tau$.  The symbol $\rightsquigarrow$ means the change of the exponent of the Euclidean distance
between $x$ and the point $(4\xi+4\eta)/8$. The notation $<-20$ indicates that the exponent goes less than $-20$.}
\label{table:stability2}
\end{table}
This result shows that not only the approximation precision but also the location of other zeros and the tolerance may also affect the stability of the convergence.

\subsection{Comparison} In this subsection, we compare the efficiency and robustness of our new algorithm with LVZ algorithm in refining deflation-one singular zeros.

\subsubsection{Efficiency}

The efficiency comparison is conducted by applying the LVZ algorithm and ours on a random variant of the $n\times n$ system 
 \begin{equation*}
	f(x_1,x_2,x_3,\ldots,x_n) = \begin{bmatrix}
		x_1^2 &
		\cdots&
		x_k^2&
		x_{k+1}&
		\cdots&
		x_n
	\end{bmatrix}^\top.
\end{equation*}
It has a deflation-one singular zero at $\xi$ at the origin with $\kappa =k,\rho=k$, and $\mu =2^k$. Then, an $n\times n$ matrix $A$ and an $n$-tuple vector $b$ are randomly chosen to generate a variant system $f(A(X-b))$, which has a particular zero at $b$ with the same type of deflation-one singularity. 

The results comparing our algorithm to the LVZ algorithm are presented in Table \ref{table:comparisonResults}. We iterate the algorithms $3$ times and report the averaged elapsed time in seconds for one iteration. Our algorithm uses smaller matrices of size $\kappa \times \kappa$, compared to the LVZ algorithm that uses matrices of size $(2n+1) \times (2n-\brx+1)$. Therefore, as $n$ increases and $\brx$ decreases, our algorithm may become faster than the LVZ algorithm.

\begin{table}
\centering
\begin{tabular}{c|c|c|c}
    $n$ & $\brx$ &LVZ (second) & Two-step Newton (second)\\
    \hhline{=|=|=|=}
    10 & 2 & 0.067 & 0.070 \\
    10 & 8 & 0.060 & 0.063 \\
    25 & 2 & 1.433 & 0.198 \\
    25 & 23 & 1.265 & 0.171\\
    50 & 2 & 36.354 & 1.533 \\
    50 & 48 & 30.567 & 3.471
\end{tabular}
\caption{Comparison of elapsed times between LVZ algorithm and the two-step Newton's method (Algorithm \ref{revisedNewton}).}
\label{table:comparisonResults}
\end{table}

\subsubsection{Robustness}

The robustness comparison is done by applying the LVZ algorithm and ours to the following system:
  \begin{equation*}
  f(x,y) = \begin{bmatrix}
  x-y^2\\
  x^2-y^2
  \end{bmatrix}
  \end{equation*}
with a deflation-one singular zero at $\xi$ at the origin with $\kappa =1,\rho=1$ and $\mu =2$.
Let $x=[0.3,0.3]^\top$ be an approximation of $\xi$ and $\tau=0.1$ be a tolerance to identify the type of singularity.
The LVZ algorithm returns
\[g(x,y,\lambda) = \begin{bmatrix}
  f(x,y)\\
  \lambda-2y\\
  2x\lambda-2y
  \end{bmatrix}\]
as the deflated overdetermined system and $[0.3,0.3,0.7059]^\top$ as an approximation of the regular zero $[0,0,0]^\top$.
Then, Gauss-Newton's method returns 
\begin{multline*}
(0.3,0.3,0.7059)\rightarrow(0.4227,0.4718,1.0336)\rightarrow(0.5086,0.6369,1.2827)\\\rightarrow(0.5010,0.6136,1.2264)\rightarrow(0.5001,0.6124,1.2248)\rightarrow
(0.5000,0.6124,1.2247) 
\end{multline*}
after five iterations. It doesn't converge to $[0,0,0]^\top$ but to $[1/2,\sqrt{6}/4,\sqrt{6}/2]^\top$ which is a stationary point of
$\min\|g\|^2$ (see Remark \ref{rmk:deflationProperties2}).
Meanwhile, the two-step Newton's method returns 
\begin{multline*}
	(0.3,0.3)
	\rightarrow(0.3496,0.0277)\rightarrow(0.0562,0.0621)\\\rightarrow(0.0121,0.0082)\rightarrow(0.0002,0.0002)\rightarrow
	(1.3\times10^{-7},8.4\times10^{-8})
\end{multline*} after five iterations. The quadratic convergence is observed in the last two iterations.

\section{Conclusion and possible extensions}

This paper introduces the revised Newton iteration for deflation-one singular zeros of square analytic systems. We provide proof of its practicality by demonstrating that an initial approximation for a deflation-one singular zero converges quadratically to an exact root of the system. Additionally, we conducted experiments to show that our algorithm has more effectiveness than the methods exploiting an overdetermined system obtained by deflation. Lastly, our examples demonstrate that our algorithm can be more effective for larger-sized systems.

We suggest some open problems that we didn't pursue in this paper.

\begin{itemize}
    \item[\textbf{Quantitative quadratic convergence}] Smale's $\alpha$-theory \cite{Blum:1997} certifies the convergence of the Newton iteration for a given approximation. We require a quantitative quadratic convergence of the Newton iteration to obtain $\alpha$-theory type result for deflation-one singularity. In other words, for the \textit{Newton operator} $N(f,x)$ defined by $$N(f,x):=\left\{\begin{array}{ll}
        x-Df(x)^{-1}f(x) & \text{if }Df(x)\text{ is invertible}  \\
        x & \text{otherwise}
    \end{array},\right.$$ we need to show that for every $k\in \mathbb{N}$,
    \[\left\|N(f,x)^k-\xi\right\|\leq \left(\frac{1}{2}\right)^{2^k-1}\left\|x-\xi\right\|\]
    where $N(f,x)^k$ is the $k$-th iteration of the Newton operator.

   \item[\textbf{Cluster isolation for deflation-one singularity}]
    The cluster isolation problem has been studied in \cite{Dedieu2001on,HJLZ2020}, where an invertible operator was developed for singular zeros, and a region isolating the singular zero from other zeros of the system was derived. Since the operator $\opA$ in our two-step Newton's method plays a similar role, it may be possible to extend the method to approach cluster isolation for deflation-one singular zeros.

    \item[\textbf{Toward higher deflation steps}] 
    Extending the algorithm to isolated singular zeros that require multiple deflation steps or higher-order deflation (e.g., see \cite{LVZ:2008}) is a natural direction for future work. The operator $\opA$ must be modified accordingly to handle such singular zeros. Specifically, a square invertible operator must be derived from the augmented system obtained by the deflation method. Developing regularizing algorithms for singular zeros with multiple deflation steps or higher-order deflation would also be essential for understanding multiple zeros that have not been studied yet.

    \item[\textbf{Refining iteration for the singular value decomposition}] As Algorithm \ref{alg1} requires the singular value decomposition, iteratively refining the decomposition may result in faster computation. Considering that Newton's method is often applied iteratively, the refining iteration can be induced by reusing the decomposition process from the previous Newton step.
\end{itemize}

\section*{Acknowledgement}
Agnes Szanto (1966-2022) was an advocate and enthusiastic proponent of hybrid symbolic-numeric computations. She dedicated much of her career to describing and certifying singular solutions of polynomial systems \cite{HMS2015, HMS2017, AHZ2018, MMS2020, mantzaflaris2023certified}, and was widely regarded as one of the most active experts in this area. It is with great sadness that we acknowledge her passing, which represents a significant loss to our community. As a tribute to her memory, we dedicate this paper to Agnes Szanto and will remember her for her enthusiasm, generosity, and kindness.

\bibliography{paper}

\begin{thebibliography}{OWM83}

\bibitem[AHS18]{AHZ2018}
Tulay~Ayyildiz Akoglu, Jonathan~D. Hauenstein, and Agnes Szanto.
\newblock Certifying solutions to overdetermined and singular polynomial
  systems over $\mathbb{Q}$.
\newblock {\em Journal of Symbolic Computation}, 84:147 -- 171, 2018.

\bibitem[BCSS98]{Blum:1997}
Lenore Blum, Felipe Cucker, Michael Shub, and Steve Smale.
\newblock {\em Complexity and Real Computation}.
\newblock Springer-Verlag New York, Inc., Secaucus, NJ, USA, 1998.

\bibitem[BL21]{burr2021inflation}
Michael Burr and Anton Leykin.
\newblock Inflation of poorly conditioned zeros of systems of analytic
  functions.
\newblock {\em Arnold Mathematical Journal}, 7(3):431--440, 2021.

\bibitem[BLL23]{burr2023isolating}
Michael Burr, Kisun Lee, and Anton Leykin.
\newblock Isolating clusters of zeros of analytic systems using
  arbitrary-degree inflation.
\newblock In {\em Proceedings of the 2023 International Symposium on Symbolic
  and Algebraic Computation}, ISSAC '23, page 126–134, New York, NY, USA,
  2023. ACM.

\bibitem[DK80a]{DeckerKelley:1980I}
D.~W. Decker and C.~T. Kelley.
\newblock {Newton}'s method at singular points {I}.
\newblock {\em SIAM Journal on Numerical Analysis}, 17:66--70, 1980.

\bibitem[DK80b]{DeckerKelley:1980II}
D.~W. Decker and C.~T. Kelley.
\newblock {Newton}'s method at singular points {II}.
\newblock {\em SIAM Journal on Numerical Analysis}, 17:465--471, 1980.

\bibitem[DK82]{DeckerKelley:1982}
D.~W. Decker and C.~T. Kelley.
\newblock Convergence acceleration for {Newton's} method at singular points.
\newblock {\em SIAM Journal on Numerical Analysis}, 19:219--229, 1982.

\bibitem[DLZ11]{DLZ:2009}
B.~Dayton, T.~Li, and Z.~Zeng.
\newblock Multiple zeros of nonlinear systems.
\newblock {\em Mathematics of Computation}, 80:2143--2168, 2011.

\bibitem[DS01]{Dedieu2001on}
J.~P. Dedieu and M.~Shub.
\newblock On simple double zeros and badly conditioned zeros of analytic
  functions of $n$ variables.
\newblock {\em Mathematics of Computation}, 70(233):319--327, 2001.

\bibitem[DZ05]{DZ:2005}
B.~Dayton and Z.~Zeng.
\newblock Computing the multiplicity structure in solving polynomial systems.
\newblock In M.~Kauers, editor, {\em Proceedings of the 2005 International
  Symposium on Symbolic and Algebraic Computation}, pages 116--123, New York,
  NY, USA, 2005. ACM.

\bibitem[Fau99]{faugere1999new}
Jean-Charles Faugere.
\newblock A new efficient algorithm for computing {G}r{\"o}bner bases ({F}4).
\newblock {\em Journal of Pure and Applied Algebra}, 139(1-3):61--88, 1999.

\bibitem[GLSY05]{giusti2005location}
Marc Giusti, Gr{\'e}goire Lecerf, Bruno Salvy, and J-C Yakoubsohn.
\newblock On location and approximation of clusters of zeros of analytic
  functions.
\newblock {\em Foundations of Computational Mathematics}, 5:257--311, 2005.

\bibitem[GLSY07]{giusti2007location}
Marc Giusti, Gr{\'e}goire Lecerf, Bruno Salvy, and J-C Yakoubsohn.
\newblock On location and approximation of clusters of zeros: case of embedding
  dimension one.
\newblock {\em Foundations of Computational Mathematics}, 7:1--58, 2007.

\bibitem[GO81]{GriewankOsborne:1981}
Andreas Griewank and M.~R. Osborne.
\newblock {Newton's} method for singular problems when the dimension of the
  null space is $>1$.
\newblock {\em SIAM Journal on Numerical Analysis}, 18:145--149, 1981.

\bibitem[Gri85]{Griewank85}
A.~Griewank.
\newblock On solving nonlinear equations with simple singularities or nearly
  singular solutions.
\newblock {\em SIAM Review}, 27(4):537--563, 1985.

\bibitem[GY20]{giusti2020approximation}
Marc Giusti and Jean-Claude Yakoubsohn.
\newblock Approximation num{\'e}rique de racines isol{\'e}es multiples de
  syst{\`e}mes analytiques.
\newblock {\em Annales Henri Lebesgue}, 3:901--957, 2020.

\bibitem[HJLZ20]{HJLZ2020}
Zhiwei Hao, Wenrong Jiang, Nan Li, and Lihong Zhi.
\newblock On isolation of simple multiple zeros and clusters of zeros of
  polynomial systems.
\newblock {\em Mathematics of Computation}, 89(322):879--909, 2020.

\bibitem[HMS15]{HMS2015}
Jonathan~D. Hauenstein, Bernard Mourrain, and Agnes Szanto.
\newblock Certifying isolated singular points and their multiplicity structure.
\newblock In {\em Proceedings of the 2015 ACM on International Symposium on
  Symbolic and Algebraic Computation}, ISSAC '15, pages 213--220, New York, NY,
  USA, 2015. ACM.

\bibitem[HMS17]{HMS2017}
Jonathan~D Hauenstein, Bernard Mourrain, and Agnes Szanto.
\newblock On deflation and multiplicity structure.
\newblock {\em Journal of Symbolic Computation}, 83:228--253, 2017.

\bibitem[HW13]{Hauenstein2013}
Jonathan~D. Hauenstein and Charles~W. Wampler.
\newblock Isosingular sets and deflation.
\newblock {\em Foundations of Computational Mathematics}, 13(3):371--403, 2013.

\bibitem[KSS98]{kobayashi1998numerical}
Hidetsune Kobayashi, Hideo Suzuki, and Yoshihiko Sakai.
\newblock Numerical calculation of the multiplicity of a solution to algebraic
  equations.
\newblock {\em Mathematics of Computation}, 67(221):257--270, 1998.

\bibitem[KY88]{Kung}
Konstantinos Konstantinides and Kung Yao.
\newblock Statistical analysis of effective singular values in matrix rank
  determination.
\newblock {\em IEEE Transactions on Acoustics Speech and Signal Processing},
  36(5):757--763, 1988.

\bibitem[Lec02]{Lecerf:2002}
G.~Lecerf.
\newblock Quadratic {N}ewton iteration for systems with multiplicity.
\newblock {\em Foundations of Computational Mathematics}, 2(3):247--293, 2002.

\bibitem[Ley08]{leykin2008numerical}
Anton Leykin.
\newblock Numerical primary decomposition.
\newblock In {\em Proceedings of the twenty-first international symposium on
  Symbolic and algebraic computation}, pages 165--172, 2008.

\bibitem[LVZ06]{leykin2006newton}
A.~Leykin, J.~Verschelde, and A.~Zhao.
\newblock Newton's method with deflation for isolated singularities of
  polynomial systems.
\newblock {\em Theoretical Computer Science}, 359(1-3):111--122, 2006.

\bibitem[LVZ08]{LVZ:2008}
Anton Leykin, Jan Verschelde, and Ailing Zhao.
\newblock Higher-order deflation for polynomial systems with isolated singular
  solutions.
\newblock In {\em Algorithms in Algebraic Geometry}, volume 146 of {\em The IMA
  Volumes in Mathematics and its Applications}, pages 79--97. Springer New
  York, 2008.

\bibitem[LZ12]{LZ:2011}
Nan Li and Lihong Zhi.
\newblock Computing isolated singular solutions of polynomial systems: {case}
  of breadth one.
\newblock {\em SIAM Journal on Numerical Analysis}, 50(1):354--372, 2012.

\bibitem[LZ22]{li2022improved}
Nan Li and Lihong Zhi.
\newblock Improved two-step {N}ewton's method for computing simple multiple
  zeros of polynomial systems.
\newblock {\em Numerical Algorithms}, pages 1--32, 2022.

\bibitem[Map18]{maple1994waterloo}
Maplesoft.
\newblock Maple (2018).
\newblock {\em A division of Waterloo Maple Inc., Waterloo, Ontario}, 2018.

\bibitem[MK99]{moritsugu1999multiple}
Schuichi Moritsugu and Kazuko Kuriyama.
\newblock On multiple zeros of systems of algebraic equations.
\newblock In {\em Proceedings of the 1999 international symposium on Symbolic
  and algebraic computation}, pages 23--30. ACM, 1999.

\bibitem[MM11]{MM:2011}
Angelos Mantzaflaris and Bernard Mourrain.
\newblock Deflation and certified isolation of singular zeros of polynomial
  systems.
\newblock In {\em Proceedings of the 36th International Symposium on Symbolic
  and Algebraic Computation}, ISSAC '11, pages 249--256, New York, NY, USA,
  2011. ACM.

\bibitem[MMS20]{MMS2020}
Angelos Mantzaflaris, Bernard Mourrain, and Agnes Szanto.
\newblock Punctual {H}ilbert scheme and certified approximate singularities.
\newblock In {\em Proceedings of the 45th International Symposium on Symbolic
  and Algebraic Computation}, ISSAC '20, page 336–343, New York, NY, USA,
  2020. Association for Computing Machinery.

\bibitem[MMS23]{mantzaflaris2023certified}
Angelos Mantzaflaris, Bernard Mourrain, and Agnes Szanto.
\newblock A certified iterative method for isolated singular roots.
\newblock {\em Journal of Symbolic Computation}, 115:223--247, 2023.

\bibitem[Oji87]{Ojika:1987}
Takeo Ojika.
\newblock Modified deflation algorithm for the solution of singular problems.
  {I}. {A} system of nonlinear algebraic equations.
\newblock {\em Journal of Mathematical Analysis and Applications}, 123(1):199
  -- 221, 1987.

\bibitem[OWM83]{OWM:1983}
Takeo Ojika, Satoshi Watanabe, and Taketomo Mitsui.
\newblock Deflation algorithm for the multiple roots of a system of nonlinear
  equations.
\newblock {\em Journal of Mathematical Analysis and Applications}, 96(2):463 --
  479, 1983.

\bibitem[Ral66]{Rall66}
L.B. Rall.
\newblock Convergence of the {N}ewton process to multiple solutions.
\newblock {\em Numerische Mathematik}, 9(1):23--37, 1966.

\bibitem[Red78]{Reddien:1978}
G.~W. Reddien.
\newblock On {Newton's} method for singular problems.
\newblock {\em SIAM Journal on Numerical Analysis}, 15(5):993--996, 1978.

\bibitem[Red79]{Reddien:1980}
G.~W. Reddien.
\newblock {Newton's} method and high order singularities.
\newblock {\em Computers {\&} Mathematics with Applications}, 5(2):79 -- 86,
  1979.

\bibitem[Ste04]{stetter2004numerical}
Hans~J Stetter.
\newblock {\em Numerical polynomial algebra}.
\newblock SIAM, 2004.

\bibitem[Stu02]{sturmfels2002solving}
B.~Sturmfels.
\newblock {\em Solving systems of polynomial equations}.
\newblock Number 97 in CBMS Reginal Conference Series in Mathematics. American
  Mathematical Soc., 2002.

\bibitem[SY05]{ShenYpma05}
Yun-Qiu Shen and Tjalling~J. Ypma.
\newblock Newton's method for singular nonlinear equations using approximate
  left and right nullspaces of the {Jacobian}.
\newblock {\em Applied Numerical Mathematics}, 54(2):256 -- 265, 2005.
\newblock 6th IMACS International Symposium on Iterative Methods in Scientific
  Computing.

\bibitem[SY07]{ShenYpma2007}
Yun-Qiu Shen and Tjalling~J. Ypma.
\newblock Solving rank-deficient separable nonlinear equations.
\newblock {\em Appl. Numer. Math.}, 57(5-7):609--615, May 2007.

\bibitem[VV22]{verschelde2022locating}
Jan Verschelde and Kylash Viswanathan.
\newblock Locating the closest singularity in a polynomial homotopy.
\newblock In {\em International Workshop on Computer Algebra in Scientific
  Computing}, pages 333--352. Springer, 2022.

\bibitem[Wed72]{wedin}
Per-Åke Wedin.
\newblock Perturbation bounds in connection with singular value decomposition.
\newblock {\em BIT}, 12(1):99--111, 1972.

\bibitem[Wey12]{weyl}
Hermann Weyl.
\newblock Das asymptotische {V}erteilungsgesetz der {E}igenwerte linearer
  partieller {D}ifferentialgleichungen (mit einer {A}nwendung auf die {T}heorie
  der {H}ohlraumstrahlung).
\newblock {\em Mathematische Annalen}, 74(4):441--479, 1912.

\bibitem[Yam84]{YAMAMOTONORIO:1984}
N.~Yamamoto.
\newblock Regularization of solutions of nonlinear equations with singular
  jacobian matrices.
\newblock {\em Journal of Information Processing}, 7(1):16--21, March 1984.

\end{thebibliography}
\bibliographystyle{alpha}

\end{document}